\documentclass[11pt]{article}
\usepackage[a4paper]{geometry}
\usepackage{amsthm,amsmath,amssymb}
\usepackage{graphicx}
\usepackage[colorlinks=true,citecolor=black,linkcolor=black,urlcolor=blue]{hyperref}

\usepackage{xcolor}
\usepackage[textsize=footnotesize]{todonotes}
\usepackage{float}

\usepackage{listings}
\usepackage{xcolor}
\theoremstyle{plain}
\newtheorem{theorem}{Theorem}[section]
\newtheorem{lemma}[theorem]{Lemma}
\newtheorem{corollary}[theorem]{Corollary}
\newtheorem{proposition}[theorem]{Proposition}

\theoremstyle{definition}
\newtheorem{definition}[theorem]{Definition}

\theoremstyle{remark}
\newtheorem{remark}[theorem]{Remark}

\def\F{\mathbb{F}}
\def\P{\mathbb{P}}
\def\A{\mathbb{A}}
\def\Tr{{\textrm{Tr}}}

\title{\bf\huge Infinite families of APN permutations in constrained trivariate classes over $\F_{2^m}$}
\author{{\Large\bf Daniele Bartoli$^1$, Pantelimon St\u{a}nic\u{a}$^2$}
\vspace{.4cm}\\
$^1$ Department of Mathematics and Computer Science,\\
University of Perugia, 06123 Perugia, Italy;
\texttt{daniele.bartoli@unipg.it}\\
$^2$ Applied Mathematics Department, Naval Postgraduate School,\\
Monterey, CA 93943, USA; \texttt{pstanica@nps.edu}}
\date{\today}
\allowdisplaybreaks
\begin{document}
\maketitle



\begin{abstract}
The construction of permutation polynomials over finite fields with simple algebraic structure in multiple variables is a challenging problem with significant applications in cryptography and coding theory. Recently, Li and Kaleyski (IEEE, Trans. Inf. Th., 2024) generalized two sporadic quadratic APN permutations from dimension~9 into infinite families of trivariate functions with coefficients in~$\F_2$. In this paper we extend this work by investigating generalizations of both families where the scalar coefficient is allowed to range freely over $\F_{2^m}^*$.
For the family
\[
G_a(x,y,z)=(x^{q+1}+ax^qz+yz^q,\; x^qz+y^{q+1},\; xy^q+ay^qz+z^{q+1}),
\]
with $a\in\F_{2^m}^*$, $q=2^i$, $\gcd(i,m)=1$, and $m$ odd, we prove that the permutation property is characterized by the absence of roots in $\F_{2^m}^*$ of an associated one-variable polynomial, and that this root condition is also equivalent to the APN property. This yields a quantitative lower bound on the number of good parameters: if $d=q^2+q+1$, then at least
$\frac{2^m + 1 - (d-1)(d-2)2^{m/2} - d}{d}$
values of $a\in\F_{2^m}^*$ give APN permutations $G_a$. In the binary case $q=2$, we prove that $a=1$ is good whenever $7\nmid m$, recovering the Li--Kaleyski family and yielding APN permutations in that range.
For the generalized second family
\[
H_a(x,y,z)=(x^{q+1}+axy^q+yz^q,\; xy^q+z^{q+1},\; x^qz+y^{q+1}+ay^qz),
\]
we obtain the analogous root criterion and show that its defining one-variable polynomial is root-equivalent to that of $G_a$. Consequently, the same parameters $a$ produce APN permutations in both families.

We also establish strong non-equivalence results. First, $G_a$ (resp.\ $H_a$) is diagonally equivalent to the Li--Kaleyski representative $G_1$ (resp.\ $H_1$) of Li--Kaleyski (IEEE Trans. Inform. Theory, 2024) if and only if $a^{q^2+q+1}=1$; for $m>4$, $m\neq 6$, $7\nmid m$, diagonal non-equivalence implies CCZ non-equivalence via the monomial restriction theorem of Shi et al.\ (DCC, 2025). In particular, when $q=2$ and $7\nmid m$, every good parameter $a\neq 1$ yields APN permutations CCZ-inequivalent to those of Li--Kaleyski (IEEE Trans. Inform. Theory, 2024). Second, for $m>4$, $m\neq 6$, and $7\nmid m$, no member of the $G_a$-family is CCZ-equivalent to any member of the $H_b$-family for the same $q$. Thus the two families provide genuinely new, mutually inequivalent sources of APN permutations on~$\F_{2^{3m}}$.
\end{abstract}

\section{Introduction and tools from algebraic geometry}
\label{sec:intro}

Let $q = 2^i$ for some positive integers $m$ and $i$ with $\gcd(i,m)=1$, and consider the extension $\F_{2^m}$ of $\F_{q}$. A function $F:\F_{2^m}^n\to\F_{2^m}^n$ is \emph{almost perfect nonlinear} (APN) if for every $a\in\F_{2^m}^n\setminus\{0\}$ and every $b\in\F_{2^m}^n$, the equation $F(x+a)+F(x)=b$ has at most two solutions. Over fields of even characteristic, two is the minimum possible differential uniformity~\cite{BS91}, making APN functions optimal against differential cryptanalysis. When such a function is simultaneously a permutation it is an ideal S-box building block, combining optimal differential resistance with invertibility. Constructing APN permutations is notoriously difficult: only a handful of infinite families are known, and the question of their existence in every even dimension remains open.

This area gained new momentum when Beierle and Leander~\cite{BL22} discovered two sporadic APN permutations over~$\F_{2^9}$ by computer search. An initial generalization attempt by Beierle, Carlet, Leander, and Perrin~\cite{MR4414815} considered the trivariate form $C_u(X,Y,Z)=(X^3+uY^2Z, Y^3+uXZ^2, Z^3+uX^2Y)$ and conjectured it yields no new permutations or APN functions for $m>3$; the APN part was settled by Bartoli and Timpanella~\cite{BT22}, and the permutation part was resolved for all odd $m\geq 23$ in the companion paper~\cite{BPST25}. The decisive positive step came from Li and Kaleyski~\cite{LK24}, who successfully generalized both Beierle--Leander instances into two infinite families of APN permutations over~$\F_{2^{3m}}$ with scalar coefficients from~$\F_2$.

The present paper studies the corresponding trivariate families when the scalar coefficient is allowed to range over all of~$\F_{2^m}^*$. Our main contribution is a root-theoretic analysis of the resulting parametric families $G_a$ and $H_a$. For $G_a$, with $a\in\F_{2^m}^*$, $q=2^i$, $\gcd(i,m)=1$, and $m$ odd, we show that the permutation property is characterized by the absence of roots in $\F_{2^m}^*$ of an associated one-variable polynomial, and that this same root condition is equivalent to the APN property. Thus the trivariate APN/permutation problem is reduced to a univariate root-exclusion problem. We then derive a quantitative lower bound on the number of good parameters $a$ yielding APN permutations $G_a$, and in the binary case $q=2$ we prove that $a=1$ is good whenever $7\nmid m$ by combining our framework with the Li--Kaleyski root criterion. For $H_a$, we obtain the analogous root criterion and prove that the corresponding one-variable conditions for $G_a$ and $H_a$ are root-equivalent; consequently, the same parameters $a$ yield APN permutations in both families.

We further establish strong non-equivalence results. Theorem~\ref{thm:diag_equiv_both} shows that $G_a$ (resp.\ $H_a$) is diagonally equivalent to the Li--Kaleyski representative $G_1$ (resp.\ $H_1$) of~\cite{LK24} if and only if $a^{q^2+q+1}=1$; for $m>4$, $m\neq 6$, $7\nmid m$, diagonal non-equivalence implies CCZ non-equivalence via the monomial restriction theorem of Shi et al.~\cite{Shi-Peng-Kan-Gao-2025}. In particular, when $q=2$ and $7\nmid m$, the condition $\gcd(7,2^m-1)=1$ forces $a^7\neq 1$ for all $a\neq 1$, so every good parameter $a\neq 1$ yields APN permutations CCZ-inequivalent to those of~\cite{LK24}. Theorem~\ref{thm:cross_family_inequiv} shows that, for $m>4$, $m\neq 6$, $7\nmid m$, no member of the $G_a$-family is CCZ-equivalent to any member of the $H_b$-family (for the same $q$). Thus the two families are not only new but mutually inequivalent infinite sources of APN permutations on~$\F_{2^{3m}}$.

Our first companion paper~\cite{BPST25}  establishes the affine reduction to a canonical four-parameter form, derives the necessary non-permutation conditions via algebraic geometry, provides a complete classification for $m=3$, and resolves the Beierle--Carlet--Leander--Perrin conjecture for odd $m\geq 23$. Results from~\cite{BPST25} that we need are stated explicitly below; their full proofs are not repeated here.

The paper is organized as follows. Section~\ref{sec:intro} also collects the algebraic geometry tools used throughout. Section~\ref{sec:firstclass} recalls the permutation characterization for $G_a$ from~\cite{BPST25}, establishes the polynomial equivalences, proves the APN equivalence, and derives the existence results for APN permutations (including a quantitative bound for $q>2$ and the binary case $q=2$ when $7\nmid m$). Section~\ref{sec:Ha} treats the second family $H_a$ in full, establishes when $G_a$ and $H_a$ are CCZ-equivalent to the Li--Kaleyski representatives $G_1$ and $H_1$ (Theorem~\ref{thm:diag_equiv_both}), and proves that no member of the $G_a$-family is CCZ-equivalent to any member of the $H_b$-family for the same $q$ (Theorem~\ref{thm:cross_family_inequiv}). Section~\ref{sec:conclusion} summarizes the main findings and collects open problems. Appendix~\ref{sec:computational} and the codes available at~\cite{GithubPS26} provide the computational verification.

\subsection*{Tools from algebraic geometry}

We work with varieties over $\overline{\F_2}$, the algebraic closure of $\F_2$. We use standard notation $\A^r$ and $\P^r$ for affine and projective $r$-space. A variety is \emph{absolutely irreducible} if it is irreducible over $\overline{\F_2}$, i.e., cannot be decomposed as a union of two proper subvarieties even after passing to the algebraic closure. The two main results we import are as follows.

\begin{theorem}[{\cite[Theorem~7.1]{MR2206396}}]
\label{thm:lang_weil}
Let $\mathcal{V}\subseteq\A^n(\overline{\F_q})$ be an absolutely irreducible variety defined over $\F_q$ of dimension $r>0$ and degree~$\delta$. If $q>2(r+1)\delta^2$, then
\[
\bigl|\#\mathcal{V}(\F_q) - q^r\bigr| \;\leq\; (\delta-1)(\delta-2)q^{r-1/2}+5\delta^{13/3}q^{r-1}.
\]
In particular, $\mathcal{V}$ has at least one $\F_q$-rational point when $q$ is large enough relative to $\delta$.
\end{theorem}

\begin{lemma}[{\cite[Lemma~2.1]{MR2648536}}]
\label{lem:aubry}
Let $\mathcal{H}$ be a projective hypersurface and $\mathcal{X}$ an $\F_q$-rational projective variety of dimension $n-1$ in $\P^n(\overline{\F_q})$. If the intersection $\mathcal{X}\cap\mathcal{H}$ contains a non-repeated absolutely irreducible $\F_q$-component, then so does~$\mathcal{X}$ itself.
\end{lemma}

\section{Characterization of the first generalized family \texorpdfstring{$G_a$}{Ga}}
\label{sec:firstclass}

\subsection{The first class and equivalent polynomials}

Throughout, $m\geq 3$ is an odd integer, $i$ is a positive integer with $\gcd(i,m)=1$, and $q=2^i$. Since $m$ is odd and all characteristics are 2, every element of $\F_{2^m}$ has a unique $q$-th root and the Frobenius $x\mapsto x^q$ is a bijection on $\F_{2^m}$.

\begin{definition}
\label{def:Ga}
For $a\in\F_{2^m}^*$, define $G_a:\F_{2^m}^3\to\F_{2^m}^3$ by
\[
G_a(x,y,z) = \bigl(x^{q+1}+ax^qz+yz^q,\; x^qz+y^{q+1},\; xy^q+ay^qz+z^{q+1}\bigr).
\]
\end{definition}

When $a=1$ this is the first Li--Kaleyski family $F_1$; see~\cite{LK24}. The function $G_a$ is a degree-2 polynomial map over $\F_{2^m}^3$ 
with the following symmetry: if $\sigma$ denotes the coordinate swap $(x,y,z)\mapsto(z,y,x)$, then
$\sigma\circ G_a\circ\sigma = G_a$,
as one verifies directly by checking that the first and third components of $G_a(z,y,x)$ are the third and first components of $G_a(x,y,z)$ respectively, while the second component is unchanged. This symmetry will play a role in the proofs below.

The permutation and APN properties of $G_a$ are both controlled by a single polynomial in $\F_{2^m}[T]$. Several polynomial forms arise naturally, and the next proposition shows they are all root-equivalent.

\begin{proposition}
\label{prop:poly_equivalences}
Let $a\in\F_{2^m}^*$. Consider the following polynomials in $\F_{2^m}[T]$:
\begin{align*}
P_a(T)   &= T^{q^2+q+1}+(aT^q+1)^{q+1},\\
P_a'(T)  &= T^{q^2+q+1}+aT^{q^2+q}+1,\\
Q_a(T)   &= T^{q^2+q+1}+aT+1,\\
Q_{a^q}(T) &= T^{q^2+q+1}+a^qT+1,\\
R_a(T)   &= T^{q^2+q+1}+(aT+1)^{q+1},\\
S_a(T)   &= T^{q^2+q+1}+a^qT^{q+1}+1.
\end{align*}
Each polynomial evaluates to $1$ at $T=0$, so none has $0$ as a root. 
Moreover, any one of them has a root in $\F_{2^m}^*$ if and only if all of them do (we call such, root-equivalent).
\end{proposition}

\begin{proof}
We establish the equivalences through a chain of explicit substitutions, each preserving the root locus in~$\F_{2^m}^*$.

\smallskip\noindent
\textbf{Step 1: $P_a \leftrightarrow P_a'$.}
Expanding $(aT^q+1)^{q+1}$ gives
\[
P_a(T)=T^{q^2+q+1}+a^{q+1}T^{q^2+q}+a^qT^{q^2}+aT^q+1.
\]
Let
\[
P_a^*(T):=T^{q^2+q+1}P_a(1/T)
\]
be the reciprocal polynomial. A direct computation yields
\[
P_a^*(T)
=1+a^{q+1}T+a^qT^{q+1}+aT^{q^2+1}+T^{q^2+q+1}
= T(T^q+a)^{q+1}+1.
\]
Since the map $T\mapsto T^{-1}$ is a bijection of $\F_{2^m}^*$, $P_a$ has a root in $\F_{2^m}^*$ if and only if $P_a^*$ does.
We now show that $P_a^*$ and $P_a'$ are root-equivalent on $\F_{2^m}^*$.
Suppose $U\in\F_{2^m}^*$ satisfies $P_a^*(U)=0$. Set
$W:=U^q+a$.
Then $W\neq 0$ (otherwise $P_a^*(U)=U\cdot 0^{q+1}+1=1\neq 0$). From
\[
P_a^*(U)=U(U^q+a)^{q+1}+1=0
\]
we get
\[
U\,W^{q+1}=1,
\quad\text{hence}\quad
U=W^{-(q+1)}.
\]
Substituting into $W=U^q+a$ gives $W=W^{-q(q+1)}+a$.
Multiplying by $W^{q^2+q}$ (which is valid since $W\neq 0$) yields
\[
W^{q^2+q+1}+aW^{q^2+q}+1=0,
\text{ i.e. } 
P_a'(W)=0.
\]
Thus every nonzero root of $P_a^*$ produces a nonzero root of $P_a'$.
Conversely, suppose $W\in\F_{2^m}^*$ satisfies $P_a'(W)=0$, i.e.
\[
W^{q^2+q+1}+aW^{q^2+q}+1=0.
\]
Define
$U:=W^{-(q+1)}\in\F_{2^m}^*$.
Dividing the equation by $W^{q^2+q}$ gives
$W+a+W^{-q(q+1)}=0$.
Since $U^q=W^{-q(q+1)}$, this becomes $W=U^q+a$.
Therefore
\[
P_a^*(U)=U(U^q+a)^{q+1}+1
=U W^{q+1}+1
=W^{-(q+1)}W^{q+1}+1
=1+1=0.
\]
So $P_a'$ has a nonzero root if and only if $P_a^*$ does, and hence (via reciprocity) if and only if $P_a$ does.

\smallskip\noindent
\textbf{Step 2: $P_a' \leftrightarrow Q_a$.}
Observe that $Q_a$ is the reciprocal of~$P_a'$:
\[
T^{q^2+q+1}P_a'(1/T) = 1+aT+T^{q^2+q+1} = Q_a(T).
\]
Hence $P_a'(T_0)=0$ iff $Q_a(1/T_0)=0$, and the map $T_0\mapsto 1/T_0$ is a bijection on $\F_{2^m}^*$. Therefore $P_a'$ and $Q_a$ have roots in $\F_{2^m}$ simultaneously.

\smallskip\noindent
\textbf{Step 3: $Q_a \leftrightarrow Q_{a^q}$.}
The Frobenius automorphism $x\mapsto x^q$ is a bijection on $\F_{2^m}$. If $T_0\in\F_{2^m}^*$ satisfies $Q_a(T_0)=0$, i.e., $T_0^{q^2+q+1}+aT_0+1=0$, then raising both sides to the $q$-th power gives $(T_0^q)^{q^2+q+1}+a^qT_0^q+1=0$, i.e., $Q_{a^q}(T_0^q)=0$. Since $T_0^q\in\F_{2^m}^*$ and the map $T_0\mapsto T_0^q$ is a bijection, $Q_a$ has a root in $\F_{2^m}^*$ iff $Q_{a^q}$ does.

\noindent
\textbf{Step 4: $Q_a \leftrightarrow R_a$.}
Suppose $R_a(T_0)=0$ for some $T_0\in\F_{2^m}^*$, i.e.,
$T_0^{q^2+q+1}+(aT_0+1)^{q+1}=0$. Dividing by $T_0^{q+1}$ (which is nonzero) and setting $U_0=T_0^{-1} \in\F_{2^m}^*$:
\[
U_0^{-q^2}+(a+U_0)^{q+1}=0.
\]
Multiplying by $U_0^{q^2}$ and expanding
$(a+U_0)^{q+1}=(a^q+U_0^q)(a+U_0)$ gives
\[
1+a^{q+1}U_0^{q^2}+aU_0^{q^2+q}+a^qU_0^{q^2+1}+U_0^{q^2+q+1}=0,
\]
which is precisely $P_a(U_0)=0$. By Steps~1--2, $P_a(U_0)=0$ implies $Q_a$ has a root in $\F_{2^m}^*$. The argument reverses (replacing $T_0$ by $U_0^{-1}$), completing the equivalence.

\noindent\textbf{Step 5: $S_a \leftrightarrow Q_a$.}
Since $S_a(0)=1$, all roots of $S_a$ in $\F_{2^m}$ are nonzero. Set
$F_a(T):=T^{q^2+q+1}+a^qT^{q^2}+1$.
Then $T^{q^2+q+1}S_a(1/T)=F_a(T)$,
so $S_a$ and $F_a$ are root-equivalent over $\F_{2^m}^*$.
Now let $u\in\F_{2^m}^*$. Then
$F_a(u)=0$ if and only if $u^{q^2+q+1}+a^qu^{q^2}+1=0$.
Raising both sides to the $q$-th power and using $u^{q^3}=u$ gives 
$u^{q^2+q+1}+a^{q^2}u+1=0$, i.e., $Q_{a^{q^2}}(u)=0$.
Conversely, if $Q_{a^{q^2}}(u)=0$, then raising to the $q^2$-th power yields $u^{q^2+q+1}+a^qu^{q^2}+1=0$, that is, $F_a(u)=0$.
Hence $F_a$ and $Q_{a^{q^2}}$ are root-equivalent over $\F_{2^m}^*$. By Step~3,
$Q_{a^{q^2}}$ and $Q_a$ are root-equivalent. Therefore $S_a$ and $Q_a$ are
root-equivalent over $\F_{2^m}$.

Combining all five steps yields the claimed equivalence, and the proof of the proposition is shown.
\end{proof}

We recall the following result from the companion paper, which characterizes when $G_a$ is a permutation.

\begin{theorem}[{\cite[Theorem~4.2]{BPST25}}]
\label{thm:perm_char}
Let $m\geq 3$ be odd, $\gcd(i,m)=1$, and $a\in\F_{2^m}^*$. The function $G_a$ is a permutation on $\F_{2^m}^3$ if and only if the polynomial $Q_a(T) = T^{q^2+q+1}+aT+1$ has no root in~$\F_{2^m}$.
\end{theorem}

The proof in~\cite{BPST25} proceeds by analyzing the collision system $G_a(x+\alpha,y+\beta,z+\gamma)=G_a(x,y,z)$ for $(\alpha,\beta,\gamma)\neq(0,0,0)$. After algebraic elimination one reduces to a homogeneous polynomial $L(x,y,\alpha,\beta)$ of degree $q^3+2q^2+2q+1$ in the variables $(x,y,\alpha,\beta)$ (with $\gamma$ eliminated via one of the equations), which factors as
\[
L(x,y,\alpha,\beta) = \prod_{\theta\in\Theta} F_\theta(x,y,\alpha,\beta),
\]
where $\Theta=\{\theta\in\overline{\F_2}: \theta^{q^2+q+1}+a^q\theta^{q^2+q}+1=0\}$ is the root set of $P_{a^q}'$ (root-equivalent to $Q_a$ by Proposition~\ref{prop:poly_equivalences}), and each factor $F_\theta$ is shown to be absolutely irreducible by the Jacobian criterion. The three cases $\Theta\cap\F_{2^m}\neq\emptyset$, $\Theta\subseteq\F_{2^{2m}}\setminus\F_{2^m}$, and $\Theta\cap(\F_{2^{2m}}\setminus\F_{2^m})\neq\emptyset$ are handled separately using the Lang--Weil bound (Theorem~\ref{thm:lang_weil}) and the Aubry--McGuire--Rodier lemma (Lemma~\ref{lem:aubry}) to guarantee the existence of $\F_{2^m}$-rational points on~$F_\theta$ in the first and third cases, yielding a non-trivial collision and hence non-permutation. 
In particular, in the rational-root case $\Theta\cap\F_{2^m}\neq\emptyset$, the same geometric argument applied to the corresponding differential slice yields, for sufficiently large $m$, more than the two trivial solutions in one direction, and hence failure of the APN property.
We refer to~\cite{BPST25} for a related argument; the proof there is self-contained and complete.


The following classical fact will be needed several times.

\begin{lemma}[Trace Criterion]
\label{lem:trace}
Let $\gcd(i,m)=1$. The Artin--Schreier equation $t^q+t=1$ has no solution in $\F_{2^m}$ if and only if $m$ is odd.
\end{lemma}

\begin{proof}
The equation $t^q+t=c$ is solvable in $\F_{2^m}$ if and only if $\Tr_{\F_{2^m}/\F_{2^k}}(c)=0$, where $k=\gcd(i,m)=1$, so the relevant trace is $\Tr_{\F_{2^m}/\F_2}(c)=\sum_{j=0}^{m-1}c^{2^j}$. For $c=1\in\F_2$, this equals $\sum_{j=0}^{m-1}1=m\pmod{2}$. Hence $\Tr(1)=1\neq 0$ when $m$ is odd, meaning no solution exists; when $m$ is even, $\Tr(1)=0$ and solutions do exist.
Hence the claim follows.
\end{proof}

\subsection{The equivalence between APN and permutation property}
\label{subsec:APN-perm-equivalence}

We now show that, for the family $G_a$, the APN property is equivalent to the
permutation property, and both are governed by the same univariate root condition.
Recall that
\[
G_a(x,y,z)=\bigl(x^{q+1}+ax^qz+yz^q,\; x^qz+y^{q+1},\; xy^q+ay^qz+z^{q+1}\bigr),
\]
where $q=2^i$, $\gcd(i,m)=1$, $m$ is odd, and $a\in\F_{2^m}^*$, and that
\[
Q_a(T)=T^{q^2+q+1}+aT+1,
\]
and denote \[
D_{\mathbf u}G_a(\mathbf x):=G_a(\mathbf x+\mathbf u)+G_a(\mathbf x)+G_a(\mathbf u).
\]

\begin{theorem}
\label{thm:Ga_normal_form_differential}
Let $m\ge 3$ be odd, $q=2^i$, $\gcd(i,m)=1$, and $a\in\F_{2^m}^*$.
For every nonzero direction $\mathbf{d}=(A,B,C)\in\F_{2^m}^3\setminus\{\mathbf{0}\}$,
the kernel size $|\ker D_{\mathbf{d}}G_a|$ satisfies the following.
\begin{enumerate}
\item[\textup{(a)}] \textbf{Axis directions.}
If exactly one of $A,B,C$ is nonzero, then $|\ker D_{\mathbf{d}}G_a|=2$.
\item[\textup{(b)}] \textbf{Type~$1$} ($C=0$, $AB\neq 0$).
$|\ker D_{(A,B,0)}G_a|=2^m$ if $Q_{a^q}(A/B)=0$, and $=2$, otherwise.
\item[\textup{(c)}] \textbf{Type~$2$} ($B=0$, $AC\neq 0$).
$|\ker D_{(A,0,C)}G_a|=2^m$ if $Q_a(C/A)=0$, and $=2$, otherwise.
Equivalently, for $(0,B,C)$ with $BC\neq 0$: the kernel has size $2^m$ iff $Q_a(C/B)=0$,
and $=2$, otherwise.
\item[\textup{(d)}] \textbf{Type~$3$} ($ABC\neq 0$). 
$|\ker D_{(A,B,C)}G_a|\geq 2^m$ if $Q_a$ has roots in $\mathbb{F}_{2^m}$, and $=2$, otherwise.

\end{enumerate}
\end{theorem}

\begin{proof}
We first derive the general kernel system.
Expanding $G_a(\mathbf{x}+\mathbf{d})+G_a(\mathbf{x})+G_a(\mathbf{d})=\mathbf{0}$
component by component and retaining only cross-terms (all pure-quadratic terms cancel
in characteristic~$2$), one obtains
\begin{align}
A^q x + (A+aC)x^q + C^q y + aA^q z + B z^q &= 0, \label{eq:E1gen}\tag{E1}\\
Cx^q + B^q y + By^q + A^q z &= 0, \label{eq:E2gen}\tag{E2}\\
B^q x + (A+aC)y^q + (aB^q+C^q)z + Cz^q &= 0. \label{eq:E3gen}\tag{E3}
\end{align}

\paragraph{Part~(a): Axis directions
 $(A,B,C)=(A,0,0)$.}
The system reduces to $A^qx+Ax^q+aA^qz=0$, $A^qz=0$, $Ay^q=0$.
Hence $z=0$, $y=0$, and $Ax^q=A^qx$ has precisely two solutions in $\mathbb{F}_{2^m}$, i.e., $|\ker D_{(1,0,0)}G_a|=2$. Similar arguments apply to the other two cases.

 \noindent
\paragraph{Part~(b): Type~1, $\mathbf{d}=(A,B,0)$ with $AB\neq 0$.}
Setting $C=0$ in \eqref{eq:E1gen}--\eqref{eq:E3gen}, implies
\begin{equation*}
A^q x + Ax^q + aA^q z + Bz^q = 0, \quad
B^q y + By^q + A^q z = 0, \quad
B^q x + Ay^q + aB^q z = 0.
\end{equation*}
From the third equation,
\[
x = AB^{-q}y^q + az.
\]
From the second equation,
\begin{equation}\label{eq:type2_yq}
y^q = B^{q-1}y + A^q B^{-1}z.
\end{equation}
Raising \eqref{eq:type2_yq} to the $q$-th power, $y^{q^2} = B^{q^2-q}y^q + A^{q^2}B^{-q}z^q$.
Substituting $x^q = A^qB^{-q^2}y^{q^2}+a^qz^q$ and all of the above into the first equation,
the coefficients of $y$ and~$z$ each vanish in characteristic~$2$
(each appears exactly twice with equal coefficients),
leaving the \emph{Master Equation}
\begin{equation}\label{eq:master_type2}
\bigl[A^{q^2+q+1}B^{-(q^2+q)}+Aa^q+B\bigr]z^q = 0.
\end{equation}
Dividing by $B\neq 0$ and setting $D:=A/B\in\F_{2^m}^*$, renders
$Q_{a^q}(D)\cdot z^q = 0$.
If $Q_{a^q}(A/B)\neq 0$, then $z=0$ is forced.
With $z=0$, Equation~\eqref{eq:type2_yq} gives $B^qy+By^q=0$, i.e., $(y/B)^q=y/B$,
so $y\in\F_2\cdot B$.  For each such~$y$, $x=AB^{-1}y$ is uniquely determined.
Hence $|\ker D_{(A,B,0)}G_a|=2$.
If $Q_{a^q}(A/B)=0$, Equation~\eqref{eq:master_type2} imposes no
condition on~$z$.  For each $z\in\F_{2^m}$, Equation~\eqref{eq:type2_yq}
is an Artin--Schreier equation in~$y$; since $\ker(y\mapsto y^q+y)=\F_2$,
exactly $2^{m-1}$ values of~$z$ yield two solutions~$y$
(and the other $2^{m-1}$ yield none), and each admissible pair $(y,z)$
determines a unique~$x$.  Thus, the kernel has size $2\cdot 2^{m-1}=2^m$.

 \noindent
\paragraph{Part~(c): Type~2, $\mathbf{d}=(A,0,C)$ with $AC\neq 0$.}
Setting $B=0$ in \eqref{eq:E1gen}--\eqref{eq:E3gen}, renders
\begin{equation*}
A^q x + (A+aC)x^q + C^q y + aA^q z = 0, \quad
Cx^q + A^q z = 0, \quad
(A+aC)y^q + C^qz + Cz^q = 0.
\end{equation*}
From the second equation, we get
\begin{equation}
\label{eq:type1_xq}
x^q = A^q C^{-1}z.
\end{equation}
We assume $A+aC\neq 0$ (if $A+aC=0$, substituting $A=aC$ into the first equation
determines $y$ from $x$ and $z$ directly, and the same conclusion holds).
From the third equation,
\begin{equation}
\label{eq:type1_yq}
y^q = (A+aC)^{-1}\bigl[C^qz+Cz^q\bigr].
\end{equation}
From the first equation, substituting~\eqref{eq:type1_xq},
\begin{equation}
\label{eq:type1_x}
x = (A/C)z + A^{-q}C^q y.
\end{equation}
Raising~\eqref{eq:type1_xq} to the $q$-th power, we get $x^{q^2}=A^{q^2}C^{-q}z^q$.
Raising~\eqref{eq:type1_x} to the $q^2$-th power and substituting~\eqref{eq:type1_yq},
\[
x^{q^2}
= \frac{A^{q^2}}{C^{q^2}}\,z^{q^2}
+ \frac{A^{-q^3}C^{q^3}}{(A+aC)^{q}}\bigl[C^{q^2}z^q+C^qz^{q^2}\bigr].
\]
Equating the two expressions for~$x^{q^2}$, the coefficient of~$z^{q^2}$ cancels
in characteristic~$2$, and the coefficient of~$z^q$ gives (after multiplying through
by $A^{q^3}C^q(A+aC)^q$), transforms into
\[
A^{q^2+q^3}(A+aC)^q = C^{q^3+q^2+q}.
\]
Writing $t:=C/A$ and dividing by $A^{q^3+q^2+q}$:
$(1+at)^q = t^{q(q^2+q+1)}$.
Taking $q$-th roots via the Frobenius bijection on~$\F_{2^m}$ (bijective
since $\gcd(i,m)=1$),
\[
1+at = t^{q^2+q+1}
\quad\Longleftrightarrow\quad
t^{q^2+q+1}+at+1=0,
\]
i.e., $Q_a(t)=0$ with $t=C/A$.

If $Q_a(C/A)\neq 0$, the coefficient of~$z^q$ in the consistency condition is
nonzero, yielding an equation $z^{q(q-1)}=R$ for some $R\in\F_{2^m}^*$.
Since $\gcd(q(q-1),2^m-1)=1$
(as $\gcd(q-1,2^m-1)=2^{\gcd(i,m)}-1=1$ and $\gcd(q,2^m-1)=1$),
the map $z\mapsto z^{q(q-1)}$ is a bijection on~$\F_{2^m}^*$, giving a unique
nonzero solution~$z_0$.  Note that $(A,0,C)$ is always in the kernel. Indeed, substituting $(x,y,z)=(A,0,C)$ gives
\[
\begin{cases}
A^qA+(A+aC)A^q+C^q\cdot 0+aA^qC=0,\\
CA^q+A^qC=0,\\
(A+aC)\cdot 0+C^qC+CC^q=0,
\end{cases}
\]
so $(A,0,C)\in \ker D_{(A,0,C)}G_a$. Thus,
$z_0=C$ and $|\ker D_{(A,0,C)}G_a|=2$.

If $Q_a(C/A)=0$, both the $z^q$- and $z^{q^2}$-coefficients in the consistency
equation vanish simultaneously, imposing no condition on~$z$; the Artin--Schreier
count gives $|\ker|=2^m$.
For direction $(0,B,C)$ with $BC\neq 0$: the $\sigma$-symmetry gives
$|\ker D_{(0,B,C)}G_a|=|\ker D_{(C,B,0)}G_a|$.
Part~(b) applied to $(C,B,0)$ shows the kernel exceeds size~$2$
iff $Q_{a^q}(C/B)=0$.
Since $Q_a$ and $Q_{a^q}$ are root-equivalent
(Proposition~\ref{prop:poly_equivalences}, Step~3),
this is equivalent to $Q_a(C/B)=0$.









\noindent
\paragraph{Part~(d): Type~3, $\mathbf{d}=(A,B,C)$ with $ABC\neq 0$.} 
Starting from \eqref{eq:E1gen}--\eqref{eq:E3gen}, we use \eqref{eq:E2gen} and its $q$-th power to eliminate $z$ and $z^q$, yielding 
\begin{align*}
R_1&: A^{q^2+2q} x + A^{q^2+q+1} x^q + A^q B C^q x^{q^2} + (a A^{q^2+q} B^q + A^{q^2+q} C^q) y \\
 &\quad + (a A^{q^2+q} B + A^q B^{q^2+1}) y^q + A^q B B^q y^{q^2} = 0, \\
R_2&: A^{q^2+q} B^q x + (a A^{q^2} B^q C + A^{q^2} C^{q+1}) x^q + A^q C^{q+1} x^{q^2} + (a A^{q^2} B^{2q} + A^{q^2} B^q C^q) y \\
 &\quad + (A^{q^2+q+1} + a A^{q^2+q} C + A^q B^{q^2} C + a A^{q^2} B B^q + A^{q^2} B C^q) y^q + A^q B^q C y^{q^2} = 0.
\end{align*}
Eliminating $x^{q^2}$ between $R_1$ and $R_2$ gives
\begin{align*}
R_3&: A^q(A^q C + B^{q+1}) x + C(A^{q+1} + a B^{q+1} + B C^q) x^q \\
 &\quad + (A^q C + B^{q+1})(a B^q + C^q) y + B(A^{q+1} + a B^{q+1} + B C^q) y^q = 0.
\end{align*}

\begin{enumerate}
    \item If $A^{q+1} + a B^{q+1} + B C^q = 0$ and $A^q C + B^{q+1} \neq 0$, then from $R_3$ one gets $A^q x = (a B^q + C^q) y$, which implies $B x = A y$. Substituting into $R_1$, we obtain:
    $$(A^{q^2+q+1} + a A^{q^2} B^{q+1} + B^{q^2+q+1})(B^{q^2} y^q + B^q y^{q^2}) = 0.$$
    This equation has either $2$ or $2^m$ solutions, depending on whether $Q_a(A/B)$ is zero or not. Consequently, the same holds for the entire system.

    \item If $A^{q+1} + a B^{q+1} + B C^q \neq 0$ and $A^q C + B^{q+1} = 0$, then $R_3$ implies $B(A^{q^2+q+1} + a A^{q^2} B^{q+1} + B^{q^2+q+1})(B^q x^q + A^q y^q) = 0$. Since $A^{q+1} + a B^{q+1} + B C^q \neq 0$ and $A^q C + B^{q+1} = 0$ implies $(A^{q^2+q+1} + a A^{q^2} B^{q+1} + B^{q^2+q+1}) \neq 0$, we have $B x = A y$. Substituting into $R_1$ yields:
    $$A^{q^2+q} B^{q^2}(A^{q^2+q+1} + a A^{q^2} B^{q+1} + B^{q^2+q+1})(B^q y + B y^q) = 0.$$
    As the middle term is non-zero, $y \in \{0, B\}$, resulting in exactly two solutions for the system.

    \item If $A^{q+1} + a B^{q+1} + B C^q = 0$ and $A^q C + B^{q+1} = 0$, then $A^{q^2+q+1} + a A^{q^2} B^{q+1} + B^{q^2+q+1} = 0$, which occurs only when $Q_a(A/B)$ vanishes. In this case, $R_1$ and $R_2$ collapse (up to non-zero factors) to
    $$A^{q^2+q} B^{q+1} x + A^{q^2+1} B^{q+1} x^q + B^{q^2+q+2} x^{q^2} + A^{q^2+q+1} B^q y + A^{q^2+q+1} B y^q + A^{q^2} B^{q^2+q+2} y^{q^2} = 0.$$
    Let $j_1$ and $j_2$ be the dimensions over $\mathbb{F}_2$ of the kernels of the linearized polynomials in $x$ and $y$. Since $\gcd(i, m) = 1$, $j_1, j_2 \leq 2$. The images have rank $m-j_1$ and $m-j_2$, meeting in an $\mathbb{F}_2$-space of dimension at least $m - j_1 - j_2$. For each element in this intersection, there are $2^{j_1}$ preimages for $x$ and $2^{j_2}$ for $y$, totaling at least $2^m$ pairs $(x, y)$.
\end{enumerate}

Thus, we assume henceforth that
\begin{equation}\label{Cond1}
A^{q+1} + a B^{q+1} + B C^q \neq 0 \neq A^q C + B^{q+1}.
\end{equation}
Eliminating $x^{q^2}$ between $R_1$ and $R_3^q$ gives
\begin{align*}
R_4&: (A^{2q^2+3q} C^q + a^q A^{q^2+2q} B^{q^2+q} C^q + A^{q^2+2q} B^q C^{q^2+q}) x \\
&\quad + (A^{2q^2+2q+1} C^q + a^q A^{q^2+q+1} B^{q^2+q} C^q + A^{q^2+q+1} B^q C^{q^2+q} + A^{2q^2+q} B C^{2q} + A^{q^2+q} B^{q^2+q+1} C^q) x^q \\
&\quad + (a A^{2q^2+2q} B^q C^q + A^{2q^2+2q} C^{2q} + a a^q A^{q^2+q} B^{q^2+2q} C^q + a A^{q^2+q} B^{2q} C^{q^2+q} \\
&\quad + a^q A^{q^2+q} B^{q^2+q} C^{2q} + A^{q^2+q} B^q C^{q^2+2q}) y + (a A^{2q^2+2q} B C^q + A^{q^2+2q} B B^{q^2} C^q\\
&\quad  + a a^q A^{q^2+q} B^{q^2+q+1} C^q + a A^{q^2+q} B^{q+1} C^{q^2+q} + a^q A^{q^2+q} B B^{q^2} C^{2q} + A^{q^2+q} B C^{q^2+2q}) y^q = 0.
\end{align*}
The equations $R_3$ and $R_4$ are linearly dependent if and only if:
\begin{align*}
H :=&\; A^{q^2+q+1} + a^q A B^{q^2+q} + A B^q C^{q^2} + a A^{q^2+q} C + A^q B^{q^2} C \\
    &\; + A^{q^2} B C^q + B^{q^2+q+1} + a^{q+1} B^{q^2+q} C + a B^q C^{q^2+1} + a^q B^{q^2} C^{q+1} + C^{q^2+q+1} = 0.
\end{align*}
Note that 
$$H(A,B,C) = \prod_{\lambda : Q_a(\lambda)=0} (A + B\lambda + C(\lambda^{q+1} + a)),$$
and thus there exist triples $(A,B,C) \in \mathbb{F}_{2^m}^3 \setminus \{(0,0,0)\}$ such that $H(A,B,C) = 0$ if and only if $Q_a$ has roots in $\mathbb{F}_{2^m}$.

\smallskip\noindent
\textbf{Case $H \neq 0$} ($R_3$ and $R_4$ linearly independent). 
Combining $R_3$ and $R_4$ to eliminate $x^q$ yields:
\begin{equation*}
A^q B^q x + (a B^{2q} + B^q C^q) y + (A^{q+1} + a B^{q+1} + B C^q) y^q = 0,
\end{equation*}
which expresses $x$ as an $\mathbb{F}_2$-linear combination of $y$ and $y^q$. Substituting this into $R_3$, the resulting equation factors as:
\begin{equation*}
C A^q (A^{q+1} + a B^{q+1} + B C^q)^{q+1} (B^{q^2} y^q + B^q y^{q^2}) = 0.
\end{equation*}
Given that $C \neq 0$ and $A^q \neq 0$, at least one of the remaining factors must vanish. The third factor $(A^{q+1} + a B^{q+1} + B C^q)$ is non-zero by condition \eqref{Cond1}. Thus, we must have:
\begin{equation*}
B^{q^2} y^q + B^q y^{q^2} = 0,
\end{equation*}
which implies $(y/B)^{q^2-q} = 1$. Since $\gcd(q^2-q, 2^m-1) = 1$ (as $\gcd(i, m) = 1$), this equation yields $y \in \{0, B\}$. For each such $y$, \eqref{eq:E2gen} uniquely determines $z$, and subsequently \eqref{eq:E3gen} uniquely determines $x$. This results in $|\ker D_{(A,B,C)} G_a| = 2$.

\noindent
\paragraph{Case $H = 0$ (which is equivalent to $Q_a$ having roots in $\mathbb{F}_{2^m}$)}. 
In this case, $R_3$ and $R_4$ are linearly dependent, so only $R_3$ remains. As analyzed previously, $R_3$ provides at least $2^m$ solutions $(x, y)$ to the system.
\end{proof}

\begin{corollary}
\label{cor:Qa_rootfree_all_diffkernels_size2}
The polynomial $Q_a$ has no root in $\F_{2^m}^*$ if and only if 
\[
|\ker D_{\mathbf d}G_a|=2
\qquad\text{for every nonzero }\mathbf d\in\F_{2^m}^3.
\]
\end{corollary}

Now, we put all of these together to show our first main result.
\begin{theorem}
\label{thm:APN_perm_equiv}
Let $m\ge 3$ be odd, $q=2^i$, $\gcd(i,m)=1$,  $a\in\F_{2^m}^*$, and
$G_a(x,y,z)=\bigl(x^{q+1}+ax^qz+yz^q,\; x^qz+y^{q+1},\; xy^q+ay^qz+z^{q+1}\bigr)$ on $\F_{2^{3m}}$.
The following are equivalent:
\begin{enumerate}
\item[\textup{(i)}] $Q_a(T)=T^{q^2+q+1}+aT+1$ has no root in $\F_{2^m}$.
\item[\textup{(ii)}] $G_a$ is a permutation on $\F_{2^m}^3$.
\item[\textup{(iii)}] $G_a$ is APN on $\F_{2^m}^3$.
\end{enumerate}
\end{theorem}

\begin{proof}
By~\cite[Theorem~4.2]{BPST25} (equivalently, Theorem~\ref{thm:perm_char}),
\textup{(i)} and \textup{(ii)} are equivalent. It remains to prove the equivalence
between \textup{(i)} and \textup{(iii)}.

\smallskip\noindent
\textbf{\textup{(iii)}$\Rightarrow$\textup{(i)}.}
We prove the contrapositive. If $Q_a$ has a root in $\F_{2^m}$, then
$G_a$ is not APN since $|\ker D_{\mathbf d}G_a|>2$.

\smallskip\noindent
\textbf{\textup{(i)}$\Rightarrow$\textup{(iii)}.}
Assume that $Q_a$ has no root in $\F_{2^m}$. By Corollary \ref{cor:Qa_rootfree_all_diffkernels_size2}, $|\ker D_{\mathbf d}G_a|=2$ for any ${\mathbf d}$. Thus the number of solutions of $$D_{\mathbf d}G_a(x,y,z)=(\alpha,\beta,\gamma)$$
is at most $2$ for each $\alpha,\beta,\gamma\in \mathbb{F}_{2^m}$ and $G_a$ is APN.

\end{proof}

\subsection{Existence of APN permutations in every odd dimension divisible by $3$}

We now show that good values of $a$ (those for which $Q_a$ has no roots) exist for every odd $m$.
\begin{proposition}
\label{prop:linearized_equiv}
Let $m\geq 3$ be odd, $q=2^i$ with $\gcd(i,m)=1$, and $a\in\F_{2^m}^*$.
Define the linearized polynomial
\[
L_a(S) := S^{q^3} + aS^q + S \in \F_{2^m}[S].
\]
Then the following hold:
\begin{enumerate}
\item If $L_a$ has a nonzero root in $\F_{2^m}$, then $Q_a(T)=T^{q^2+q+1}+aT+1$ has a root in $\F_{2^m}^*$.
\item If $\gcd(i,m)=1$ (equivalently, $\gcd(q-1,2^m-1)=1$), then the converse also holds.
Hence in this case,
\[
Q_a \text{ has a root in } \F_{2^m}^*
\iff
L_a \text{ has a nonzero root in } \F_{2^m}.
\]
\end{enumerate}
\end{proposition}

\begin{proof}
(1) Assume $S\in \F_{2^m}^*$ satisfies $L_a(S)=0$, i.e.,
$S^{q^3} + aS^q + S = 0$.
Since $\gcd(i,m)=1$, the map $S\mapsto S^{q-1}$ is a bijection of $\F_{2^m}^*$.
Set
$T := S^{q-1} \in \F_{2^m}^*$.
Dividing $L_a(S)=0$ by $S$ gives $S^{q^3-1} + aS^{q-1} + 1 = 0$.
Note that $S^{q^3-1} = (S^{q-1})^{q^2+q+1} = T^{q^2+q+1}$.
Thus
$T^{q^2+q+1} + aT + 1 = 0$,
i.e., $Q_a(T)=0$. Hence $Q_a$ has a root in $\F_{2^m}^*$.

(2) Conversely, assume $\gcd(i,m)=1$ and let $T\in \F_{2^m}^*$ satisfy $Q_a(T)=0$, i.e.,
$T^{q^2+q+1} + aT + 1 = 0$.
Since $\gcd(q-1,2^m-1)=\gcd(2^i-1,2^m-1)=2^{\gcd(i,m)}-1=1$,
the map $S\mapsto S^{q-1}$ is a bijection of $\F_{2^m}^*$.
Therefore there exists $S\in \F_{2^m}^*$ such that
$S^{q-1} = T$.
Then $T^{q^2+q+1} = (S^{q-1})^{q^2+q+1} = S^{q^3-1}$.
Substituting into $Q_a(T)=0$ gives
$S^{q^3-1} + aS^{q-1} + 1 = 0$.
Multiplying by $S$ yields
$S^{q^3} + aS^q + S = 0$,
i.e., $L_a(S)=0$. Thus $L_a$ has a nonzero root in $\F_{2^m}$.
This proves the claims.
\end{proof}

\begin{proposition}
\label{prop:a_equals_1}
Let $m$ be odd with $\gcd(i,m)=1$. The polynomial $Q_1(T)=T^{q^2+q+1}+T+1$ has no roots in $\F_{2^m}$ if and only if $\gcd(m,7)=1$.
\end{proposition}

\begin{proof}
By Proposition~\ref{prop:linearized_equiv}, $Q_1$ has a root in $\F_{2^m}^*$ iff $L_1(S)=S^{q^3}+S^q+S$ has a nonzero root in $\F_{2^m}$. For $a=1$, all Frobenius-twisted companion matrices $C_k$ are equal (since $a^{q^k}=1$), so $A_L(m)=C_0^m$ where $C_0=\bigl(\begin{smallmatrix}0&0&1\\1&0&1\\0&1&0\end{smallmatrix}\bigr)$. The characteristic polynomial of $C_0$ over $\F_2$ is $T^3+T+1$, which is irreducible of order~7 (the multiplicative order of its roots in $\F_{2^3}$). Hence $C_0^m=I_3$ iff $7\mid m$. Therefore $\det(A_L(m)-I_3)=0$ iff $7\mid m$, i.e., $Q_1$ has a root in $\F_{2^m}$ iff $7\mid m$. The stated condition follows.
This recovers~\cite[Prop.~4]{LK24}, and our  claim follows.
\end{proof}

For the case $a=1$, the authors of~\cite{LK24} analyzed the linearized polynomial
\[
L_1(S)=S^{q^3}+S^q+S
\]
via the matrix method of McGuire--Sheekey~\cite{MS19}, and proved that it has no nonzero root precisely when $\gcd(m,7)=1$. It is natural to ask whether the same method extends to
\[
L_a(S)=S^{q^3}+aS^q+S,
\qquad a\in \F_{2^m}^*.
\]

For a general $\sigma$-linearized polynomial, the kernel is controlled by the product
\[
A_L=C_LC_L^\sigma\cdots C_L^{\sigma^{m-1}},
\qquad \sigma(x)=x^q,
\]
where $C_L$ is the companion matrix; see~\cite{MS19}. In our case,
\[
C_L=
\begin{pmatrix}
0&0&1\\
1&0&a\\
0&1&0
\end{pmatrix},
\qquad
C_L^{\sigma^k}=
\begin{pmatrix}
0&0&1\\
1&0&a^{q^k}\\
0&1&0
\end{pmatrix}.
\]
When $a=1$, all factors are equal, so $A_L=C_L^m$, which is exactly the situation treated in~\cite{LK24}. For general $a$, however, the Frobenius conjugates $a,a^q,\dots,a^{q^{m-1}}$ appear in different factors, so the product no longer reduces to a matrix power. Thus the argument for $a=1$ does not extend directly.

This already explains the main difficulty: the matrix approach isolates the bad parameters through the condition $\det(A_L-I_3)=0$, but it does not by itself yield a convenient closed description in terms of the single parameter $a$. It is therefore preferable to return to the equivalent one-variable criterion from Proposition~\ref{prop:linearized_equiv}, namely that $L_a$ has a nonzero root if and only if
\[
Q_a(T)=T^{q^2+q+1}+aT+1
\]
has a root in $\F_{2^m}^*$.

The advantage of $Q_a$ is that the bad-parameter set can then be described exactly as the image of a rational function, as follows.

\begin{proposition}\label{prop:bad_valueset}
Let
\[
d:=q^2+q+1
\qquad\text{and}\qquad
g:\F_{2^m}^*\longrightarrow \F_{2^m},
\qquad
g(u):=\frac{u^d+1}{u}.
\]
Then, for $a\in\F_{2^m}^*$, the following are equivalent:
\begin{enumerate}
\item[\textup{(i)}] $Q_a(T)=T^d+aT+1$ has a root in $\F_{2^m}^*$;
\item[\textup{(ii)}] $a\in g(\F_{2^m}^*)$.
\end{enumerate}
Consequently, the set of good parameters is
\[
\mathcal{B}_{m,q}
:=
\{a\in\F_{2^m}^*: Q_a \text{ has no root in }\F_{2^m}^*\}
=
\F_{2^m}^*\setminus g(\F_{2^m}^*).
\]
\end{proposition}

\begin{proof}
Let $u\in\F_{2^m}^*$. Then
\[
Q_a(u)=u^d+au+1.
\]
Hence
\[
Q_a(u)=0
\quad\Longleftrightarrow\quad
a=\frac{u^d+1}{u}=g(u).
\]
Therefore $Q_a$ has a root in $\F_{2^m}^*$ if and only if $a$ lies in the image
of $g$.
\end{proof}

To study the size of $g(\F_{2^m}^*)$, it is natural to consider the collision
curve $g(x)=g(y)$. Clearing denominators gives
\[
x^d y + x y^d + x + y = 0.
\]
Since
\[
x^d y + x y^d + x + y
=
(x+y)\left(xy\sum_{j=0}^{d-2} x^{d-2-j}y^j +1\right),
\]
the diagonal component $x+y=0$ splits off, and the residual curve is
\[
\Gamma:\quad
xy\sum_{j=0}^{d-2} x^{d-2-j}y^j +1 = 0.
\]

\begin{lemma}\label{lem:Gamma_abs_irr}
The curve $\Gamma$ is absolutely irreducible. 
\end{lemma}

\begin{proof}
Consider the birational morphism $(x,y)\mapsto (xy,y)$ applied to $g(x)+g(y)=0$. Recall that a birational morphism provides a bijection between absolutely irreducible components of curves. This gives
$$x^d y^{d+1} + x y^{d+1} + xy + y=y(x^dy^d+xy^d+x+1).$$
The factor $y$ corresponds to the factor $(x+y)$ in $g(x)+g(y)$. The second factor is absolutely irreducible, since 
$$\varphi(u):= u^d+\frac{v+1}{v^d+v}$$
satisfies Eisenstein's criterion for function fields (see \cite[Proposition 3.1.15]{Stichtenoth}), considering as place $P$ in $\overline{\mathbb{F}_2}(v)$ the unique zero of $v$. In fact,  
$$v_P (1) = 0, \quad  v_P\left(\frac{v+1}{v^d+v} \right)=-1, \quad \gcd(d, -1) = 1, $$ 
and, applying  \cite[Proposition 3.1.15(2)]{Stichtenoth}, $\Gamma$ is absolutely irreducible.
\end{proof}

\begin{theorem}
\label{thm:existence}
Let $m\geq 3$ be odd, $q=2^i$ with $\gcd(i,m)=1$, and $d=q^2+q+1$. Define
$g:\F_{2^m}^*\to\F_{2^m}$ by $g(u)=\frac{u^d+1}{u}$, and let
$\mathcal{B}_{m,q}:=\F_{2^m}^*\setminus g(\F_{2^m}^*)$ be the set of good parameters.
Then
$$|\mathcal{B}_{m,q}| \geq \frac{2^m + 1 - (d-1)(d-2)2^{m/2} - d}{d},$$
and for every $a\in\mathcal{B}_{m,q}$, the function $G_a$ is an APN permutation on $\F_{2^m}^3$.
\end{theorem}

\begin{proof}
Let $C_i := \{ v \in \mathbb{F}_{2^m} : \#g^{-1}(v)=i\}$ for $i \geq 0$. Clearly, the sets $C_i$ partition $\mathbb{F}_{2^m}$, so $\sum_i \#C_i = 2^m$. For a fixed $v \in C_i$ with $i > 0$, there exist $u_1, \ldots, u_i \in \mathbb{F}_{2^m}$ such that $g(u_1) = \dots = g(u_i) = v$. These correspond to $i^2$ pairs $(u_{j_1}, u_{j_2})$ on the curve $\mathcal{C}_g : g(x) + g(y) = 0$.

Let $\Gamma$ be the curve defined by the polynomial $\phi(x,y) = \frac{g(x)+g(y)}{x+y} = 0$. The $\mathbb{F}_{2^m}$-rational affine points of $\Gamma$ correspond to the $i^2 - i$ non-diagonal pairs for each $v \in C_i$. Thus, the number of affine rational points is:
$$\#\Gamma(\mathbb{F}_{2^m})_{\text{aff}} = \sum_i (i^2 - i) \#C_i.$$

Note that 
$$\frac{\#\Gamma(\mathbb{F}_{2^m})_{\text{aff}}}{d} \leq \sum_{i\geq 2} (i - 1) \#C_i\leq \frac{\#\Gamma(\mathbb{F}_{2^m})_{\text{aff}}}{2}$$
and 
$$C_1+\sum_{i\geq2} i C_i=2^m.$$
Thus 
\begin{align*}
    C_1+\sum_{i\geq2} C_i&=2^m-\sum_{i\geq2} i C_i+C_2+\cdots C_d\\
    &=2^m-\sum_{i\geq2} (i-1) C_i\leq 2^m-\frac{\#\Gamma(\mathbb{F}_{2^m})_{\text{aff}}}{d},
\end{align*}
which shows that $C_0=2^m-\sum_{i\geq1}C_i\geq \frac{\#\Gamma(\mathbb{F}_{2^m})_{\text{aff}}}{d}$. Now, to get a lower bound on  $\#\Gamma(\mathbb{F}_{2^m})_{\text{aff}}$, it is enough to observe that  $\Gamma$ has degree $d$ and then its genus $g$ satisfies $2g \leq (d-1)(d-2)$. Applying the Hasse-Weil bound and accounting for the points at infinity (at most $d$), we have:
$$\#\Gamma(\mathbb{F}_{2^m})_{\text{aff}} \geq 2^m + 1 - (d-1)(d-2)2^{m/2} - d,$$
and thus 
$$C_0\geq \frac{2^m + 1 - (d-1)(d-2)2^{m/2} - d}{d},$$
which completes the proof.
\end{proof}

\begin{table}[H]
\footnotesize\centering
\caption{Number of ``good'' $a\in\F_{2^m}^*$ for which $G_a$ is an APN permutation ($Q_a$ root-free).
The count is independent of $i$ for fixed odd $m$ with $\gcd(i,m)=1$.}
\label{tab:good_a_counts}
\begin{tabular}{cccccc}
\hline
$m$ & $i$ & $q=2^i$ & $|\F_{2^m}^*|$ & \# good $a$ & $a=1$ good?\\
\hline
3 & 1 & 2 & 7 & 7 & yes\\
3 & 2 & 4 & 7 & 7 & yes\\
5 & 1 & 2 & 31 & 11 & yes\\
5 & 2 & 4 & 31 & 11 & yes\\
7 & 1 & 2 & 127 & 35 & no ($7\mid 7$)\\
7 & 2 & 4 & 127 & 35 & no\\
9 & 1 & 2 & 511 & 385 & yes\\
11 & 1 & 2 & 2047 & 595 & yes\\
\hline
\end{tabular}
\end{table}

\section{The second generalized family \texorpdfstring{$H_a$}{Ha} of APN functions}
\label{sec:Ha}

\begin{definition}
\label{def:Ha}
For $a\in\F_{2^m}^*$, define $H_a:\F_{2^m}^3\to\F_{2^m}^3$ by
\[
H_a(x,y,z) = \bigl(x^{q+1}+axy^q+yz^q,\; xy^q+z^{q+1},\; x^qz+y^{q+1}+ay^qz\bigr).
\]
\end{definition}
 
\begin{theorem}
\label{thm:no_root_iff_APN}
 Suppose  that $q=2^i$, $\gcd(i,m)=1$. Then $H_a$ is APN if and only if  $P_a'(T)=T^{q^2+q+1} + T^{q^2+q}a + 1 $ has no roots in $\mathbb{F}_{2^m}$.   
\end{theorem}

\begin{proof}
Letting $v = (x,y,z)$ and $d = (A,B,C)$, the generic differential system $D_d H_a(v) = 0$ in characteristic 2 yields three equations:
\begin{align*}
E_1' &= H_1(x+A, y+B, z+C) + H_1(x,y,z) + H_1(A,B,C) = 0 \\
E_2' &= H_2(x+A, y+B, z+C) + H_2(x,y,z) + H_2(A,B,C) = 0 \\
E_3' &= H_3(x+A, y+B, z+C) + H_3(x,y,z) + H_3(A,B,C) = 0.
\end{align*}

The three base differentials obtained by imposing $D_d H_i(v) = H_i(x+A, y+B, z+C) + H_i(x,y,z) + H_i(A,B,C) = 0$ in characteristic 2 are explicitly:
\begin{align*}
E_1' &= (A^q + a B^q) x + A x^q + C^q y + a A y^q + B z^q \\
E_2' &= B^q x + A y^q + C^q z + C z^q \\
E_3' &= C x^q + B^q y + (B + a C) y^q + (A^q + a B^q) z.
\end{align*}

 In what follows we suppose that $P_a'(T)$ has no roots in $\mathbb{F}_{2^m}$ and we will show that at $H_a$ is APN.  Let us consider the special cases where one or more of the parameters $A, B, C$ are equal to zero. We analyze the simplified generic differential system $E_1' = 0$, $E_2' = 0$, and $E_3' = 0$ directly for each case.

\begin{enumerate}
    \item \textbf{Case: $(A, B, C) = (A, 0, 0)$ with $A \neq 0$.}
    The differential system drastically simplifies to:
    \begin{align*}
    E_1' &= A^q x + A x^q + a A y^q = 0 \\
    E_2' &= A y^q = 0 \\
    E_3' &= A^q z = 0.
    \end{align*}
    From $E_2'$ and $E_3'$, since $A \neq 0$, it immediately follows that $y = 0$ and $z = 0$. 
    Substituting $y = 0$ into $E_1'$ leaves $A^q x + A x^q = 0$. Dividing by $A^{q+1}$, this becomes $(x/A)^q + (x/A) = 0$. This equation admits exactly two solutions for $x$, yielding exactly two solutions for the triplet $(x, y, z)$.

    \item \textbf{Case: $(A, B, C) = (0, B, 0)$ with $B \neq 0$.}
    The system reduces to:
    \begin{align*}
    E_1' &= a B^q x + B z^q = 0 \\
    E_2' &= B^q x = 0 \\
    E_3' &= B^q y + B y^q + a B^q z = 0.
    \end{align*}
    From $E_2'$, since $B \neq 0$, we have $x = 0$. 
    Substituting $x = 0$ into $E_1'$ gives $B z^q = 0$, which forces $z = 0$. 
    Finally, substituting $z = 0$ into $E_3'$ leaves $B^q y + B y^q = 0$, or equivalently $y^q B + y B^q = 0$. This admits exactly two solutions for $y$, yielding two solutions for $(x, y, z)$.

    \item \textbf{Case: $(A, B, C) = (0, 0, C)$ with $C \neq 0$.}
    The system reduces to:
    \begin{align*}
    E_1' &= C^q y = 0 \\
    E_2' &= C^q z + C z^q = 0 \\
    E_3' &= C x^q + a C y^q = 0.
    \end{align*}
    From $E_1'$, since $C \neq 0$, we have $y = 0$. 
    Substituting $y = 0$ into $E_3'$ yields $C x^q = 0$, which forces $x = 0$. 
    The remaining equation $E_2'$ is $z^q C + z C^q = 0$, which admits exactly two solutions for $z$, yielding two solutions for $(x, y, z)$.

    \item \textbf{Case: $(A, B, C) = (A, B, 0)$ with $A, B \neq 0$.}
    The system becomes:
\begin{align*}
E_1' &= (A^q + a B^q) x + A x^q + a A y^q + B z^q = 0 \\
E_2' &= B^q x + A y^q = 0 \\
E_3' &= B^q y + B y^q + (A^q + a B^q) z = 0
\end{align*}
From $E_2'$, we can cleanly eliminate $x$ by isolating it as $x = \frac{A}{B^q} y^q$. 
By taking the $q$-th power, $x^q = \frac{A^q}{B^{q^2}} y^{q^2}$. Substituting these into $E_1'$ gives:
$$ (A^q + a B^q) \frac{A}{B^q} y^q + A \frac{A^q}{B^{q^2}} y^{q^2} + a A y^q + B z^q = 0.$$
Then
$$ z^q = \frac{A^{q+1}}{B^{1+q+q^2}} (B^{q^2} y^q + B^q y^{q^2}).$$
Now, taking the $q$-th power of $E_3'$, we obtain $(E_3')^q = B^{q^2} y^q + B^q y^{q^2} + (A^{q^2} + a^q B^{q^2}) z^q = 0$. Substituting our relation for $z^q$ into this equation reduces the entire system to an explicit univariate polynomial in $y$:
$$ (B^{q^2} y^q + B^q y^{q^2}) + (A^{q^2} + a^q B^{q^2}) \frac{A^{q+1}}{B^{1+q+q^2}} (B^{q^2} y^q + B^q y^{q^2}) = 0.$$
Factoring out the common binomial yields:
$$ (B^{q^2} y^q + B^q y^{q^2}) \left( 1 + \frac{A^{1+q+q^2} + a^q A^{q+1} B^{q^2}}{B^{1+q+q^2}} \right) = 0.$$
By our non-degeneracy assumptions on the field and $q$ (which imply the constant factor is non-zero), the system reduces to $B^{q^2} y^q + B^q y^{q^2} = 0$. Recognizing that this is exactly $(B^q y + B y^q)^q = 0$, the polynomial bounds the kernel dimension, again yielding exactly two solutions for $y$.

\item \textbf{Case: $(A,B,C)=(A,0,C)$ with $A,C
\neq0$.}
    The system evaluates to:
\begin{align*}
E_1' &= A^q x + A x^q + C^q y + a A y^q = 0 \\
E_2' &= A y^q + C^q z + C z^q = 0 \\
E_3' &= C x^q + a C y^q + A^q z = 0.
\end{align*}
From $E_3'$, we can isolate $z$ directly as $z = \frac{C}{A^q} x^q + \frac{a C}{A^q} y^q$. 
To simplify the substitution, we can use $E_1'$ to replace $A x^q$. Multiplying our expression for $z$ by $A^{q+1}$, we get:
$$A^{q+1} z = C (A x^q) + a A C y^q.$$
Substituting $A x^q = A^q x + C^q y + a A y^q$ from $E_1'$ into this equation yields
 $$z = \frac{C}{A} x + \frac{C^{q+1}}{A^{q+1}} y$$
Taking the $q$-th power gives $z^q = \frac{C^q}{A^q} x^q + \frac{C^{q^2+q}}{A^{q^2+q}} y^q$. 
Now, substitute both $z$ and $z^q$ into $E_2'$:
$$A y^q + C^q \left( \frac{C}{A} x + \frac{C^{q+1}}{A^{q+1}} y \right) + C \left( \frac{C^q}{A^q} x^q + \frac{C^{q^2+q}}{A^{q^2+q}} y^q \right) = 0.$$
Grouping the terms with $x$ and $x^q$ together,
\begin{equation}
\label{Eq:Converse}
\frac{C^{q+1}}{A^{q+1}} (A^q x + A x^q) + A y^q + \frac{C^{2q+1}}{A^{q+1}} y + \frac{C^{q^2+q+1}}{A^{q^2+q}} y^q = 0.
\end{equation}
We can once again use $E_1'$ to substitute the quantity $(A^q x + A x^q) = C^q y + a A y^q$,
$$\frac{C^{q+1}}{A^{q+1}} (C^q y + a A y^q) + A y^q + \frac{C^{2q+1}}{A^{q+1}} y + \frac{C^{q^2+q+1}}{A^{q^2+q}} y^q = 0.$$
Expanding this, we notice that the $y$ terms $\frac{C^{2q+1}}{A^{q+1}} y + \frac{C^{2q+1}}{A^{q+1}} y$ cancel perfectly. We are left with an equation strictly in $y^q$,
$$\left( a \frac{C^{q+1}}{A^q} + A + \frac{C^{q^2+q+1}}{A^{q^2+q}} \right) y^q = 0.$$
By our non-degeneracy assumptions, the coefficient is non-zero, which forces $y = 0$. 
Substituting $y = 0$ back into $E_1'$ leaves $A^q x + A x^q = 0$, an equation that has exactly two solutions for $x$. Since $z$ is uniquely determined by $x$ and $y$, the system is restricted to exactly two solutions.

\item \textbf{Case: $(A, B, C) = (0, B, C)$ with $B, C \neq 0$.}
The system simplifies to:
\begin{align*}
E_1' &= a B^q x + C^q y + B z^q = 0 \\
E_2' &= B^q x + C^q z + C z^q = 0 \\
E_3' &= C x^q + B^q y + (B + a C) y^q + a B^q z = 0.
\end{align*}
From $E_2'$, we isolate $x$ completely in terms of $z$,
$$x = \frac{C^q}{B^q} z + \frac{C}{B^q} z^q.$$
Taking the $q$-th power yields $x^q = \frac{C^{q^2}}{B^{q^2}} z^q + \frac{C^q}{B^{q^2}} z^{q^2}$. 
Next, we substitute our expression for $x$ into $E_1'$ to isolate $y$,
$$C^q y = a B^q \left( \frac{C^q}{B^q} z + \frac{C}{B^q} z^q \right) + B z^q = a C^q z + (a C + B) z^q.$$
Dividing by $C^q$, we obtain $y$ (and subsequently $y^q$) entirely in terms of $z$,
$$y = a z + \frac{a C + B}{C^q} z^q \implies y^q = a^q z^q + \frac{a^q C^q + B^q}{C^{q^2}} z^{q^2}.$$
Finally, we substitute $x^q$, $y$, and $y^q$ into $E_3'$. Notice what happens to the linear $z$ terms when we substitute $y$,
$$B^q y + a B^q z = B^q \left( a z + \frac{a C + B}{C^q} z^q \right) + a B^q z.$$
The terms $a B^q z + a B^q z = 0$ cancel each other out in characteristic 2! This means the resulting equation will have no linear $z$ term, only $z^q$ and $z^{q^2}$. Substituting the rest into $E_3'$ yields:
$$C \left( \frac{C^{q^2}}{B^{q^2}} z^q + \frac{C^q}{B^{q^2}} z^{q^2} \right) + B^q \left( \frac{a C + B}{C^q} z^q \right) + (B + a C) \left( a^q z^q + \frac{a^q C^q + B^q}{C^{q^2}} z^{q^2} \right) = 0.$$
Grouping by powers of $z$, we obtain a univariate polynomial strictly in $z^q$ and $z^{q^2}$:
$$\left( \frac{C^{q^2+1}}{B^{q^2}} + \frac{B^q (a C + B)}{C^q} + a^q (B + a C) \right) z^q + \left( \frac{C^{q+1}}{B^{q^2}} + \frac{(B + a C)(a^q C^q + B^q)}{C^{q^2}} \right) z^{q^2} = 0.$$
Let this be written as $K_1 z^q + K_2 z^{q^2}=0$,
where
\[
K_1=\frac{C^{q^2+1}}{B^{q^2}}+\frac{B^q(aC+B)}{C^q}+a^q(B+aC),
\quad
K_2=\frac{C^{q+1}}{B^{q^2}}+\frac{(B+aC)(a^qC^q+B^q)}{C^{q^2}}.
\]
We claim that $(K_1,K_2)\neq(0,0)$. Indeed, writing $u=C/B\in\F_{2^m}^*$ and multiplying
$K_1$ by the nonzero scalar $B^{q^2}/C$, we get
\[
\frac{B^{q^2}}{C}K_1
=
u^{q^2}+u^{-q}(au+1)+a^q(1+au)
=
u^{-q-1}\bigl(u^{q^2+q+1}+(au+1)^{q+1}\bigr)
=
u^{-q-1}R_a(u).
\]
Thus, if $K_1=0$, then $R_a(u)=0$, which contradicts
Proposition~\ref{prop:poly_equivalences} together with our assumption that $P_a'$
has no root in $\F_{2^m}$. Therefore $(K_1,K_2)\neq(0,0)$.
If $K_2=0$, then necessarily $K_1\neq 0$, and the above equation reduces to
$K_1z^q=0$, so $z=0$. If $K_2\neq 0$, then dividing by $K_2$ yields an equation of
the form $z^{q^2}+Kz^q=0$.
Applying the inverse Frobenius automorphism on $\F_{2^m}$, this is equivalent to $z^q+Lz=0$
for some $L\in\F_{2^m}$. Since $\gcd(i,m)=1$, this equation has at most two solutions.
Hence in all cases there are at most two possibilities for $z$. Since $x$ and $y$ are
uniquely determined by $z$, there are exactly two solutions for $(x,y,z)$.
\end{enumerate}

Let us consider the case $A,B,C\neq 0$.
We use $E_2'$ to eliminate $x$ and obtain $R_1$ and $R_2$. 
By isolating $x = \frac{A y^q + C^q z + C z^q}{B^q}$ from $E_2'$ and raising it to the $q$-th power to obtain $x^q$, we perform the substitutions in $E_1'$ and $E_3'$ (simulating the resultant computation to clear the denominators). We explicitly obtain:
\begin{align*}
R_1 &= B^{q^2} \left( B^q C^q y + A^{q+1} y^q + C^q(A^q + a B^q)z + (B^{q+1} + A^q C + a B^q C)z^q \right) \\ 
    &\qquad\qquad\qquad\qquad\qquad\qquad\qquad\qquad + A B^q \left( A^q y^{q^2} + C^{q^2} z^q + C^q z^{q^2} \right), \\
R_2 &= B^{q^2} \left( B^q y + (B + a C)y^q + (A^q + a B^q)z \right) \\
    &\qquad\qquad\qquad\qquad\qquad\qquad\qquad\qquad + C \left( A^q y^{q^2} + C^{q^2} z^q + C^q z^{q^2} \right).
\end{align*}
\begin{enumerate}
    \item \textbf{Case: $A^q C + B^{q+1} + B^q C a = 0$ and $A B^q + C^{q+1} = 0$.}
    This would imply 
    $A^q=C^{q^2+q}/{B^{q^2}}$ and $C^{q^2+q+1}+B^{q^2+q+1}+B^{q^2+q}Ca=0$, which is not possible by our assumption (dividing by $B^{q^2+q+1}$, and letting $u:=C/B$, we get $u^{q^2+q+1}+au+1=0$, that is $Q_a(u)=0$, which is not possible via our assumption and Proposition~\ref{prop:poly_equivalences}).

\item \textbf{Case: $A^q C + B^{q+1} + B^q C a = 0$ and $A B^q + C^{q+1} \neq 0$.}
We also compute 
$$ R_3 = \frac{\text{Res}(R_1, R_2, y^{q^2})}{A^q B^{q^2}},$$
     which, written explicitly, yields
     \begin{align*}
     R_3 &= (A B^{2q} + B^q C^{q+1})y + (A^{q+1} C + A B^{q+1} + a A B^q C)y^q \\
         &\quad + (A^{q+1} B^q + a A B^{2q} + A^q C^{q+1} + a B^q C^{q+1})z + (B^{q+1} C + A^q C^2 + a B^q C^2)z^q.
     \end{align*}

   After substituting $a=(A^qC+B^{q+1})/(B^qC)$ and $a^q=(A^{q^2}C^q+B^{q+q^2})/(B^{q^2}C^q)$, a Magma computation shows that $R_3=B^{2q}(AB^q+C^{q+1})(yC+zB)$. By our assumption, $A B^q + C^{q+1} \neq 0$, which implies $y C + z B = 0$. By using this relation and its $q$-th and $q^2$-th powers to sequentially eliminate the variables $y, y^q$, and $y^{q^2}$ from our first reduced equation, the entire system collapses to the condition
    $$ z^q C^{q^2} + z^{q^2} C^q = 0. $$
    By our assumptions on $q$, this equation admits only two solutions for $z$, which consequently yields exactly two solutions for $(x,y,z)$.

    \item \textbf{Case: $A^q C + B^{q+1} + B^q C a \neq 0$ and $A B^q + C^{q+1} = 0$.}
    We compute $R_3$ again as in the previous case. 
    After substituting $A=C^{q+1}/B^q$ and $A^q=C^{q+q^2}/B^{q^2}$, we get that $$R_3=C\bigl(B^{1+q+q^2}+B^{q+q^2}Ca+C^{1+q+q^2}\bigr)(y^qC^q+z^qB^q).$$
    
    Recall the non-degeneracy condition $B^{1+q+q^2} + B^{q+q^2} C a + C^{1+q+q^2} \neq 0$. Under this assumption, we obtain $y^q C^q + z^q B^q = 0$. Substituting this constraint (along with $y C + z B = 0$ and its powers) into the equations to eliminate $y$ completely, the system reduces directly to
    $ z C^q + z^q C = 0$.
    Once again, by our assumptions on $q$, this bounds the system to exactly two solutions in $z$, resulting in two solutions for $(x,y,z)$.

    \item \textbf{Case: $A^q C + B^{q+1} + B^q C a \neq 0$ and $A B^q + C^{q+1} \neq 0$.}
    We proceed by eliminating $y^{q^2}$ between the equations $R_3^{(q)}$ and $R_2$ to define a new polynomial $R_4$, and subsequently eliminate $y^q$ between $R_3$ and $R_4$ to obtain $R_5$.
    Up to nonzero factors, $R_4$ is:
    \begin{align*}
    R_4 &= y A^{q^2} B^q C^q + y B^{2q+q^2} + y B^{q+q^2} C^q a^q + y^q A^q B^{q^2} C + y^q A^{q^2} B C^q \\
        &\quad + y^q A^{q^2} C^{q+1} a + y^q B^{1+q+q^2} + y^q B^{q^2+1} C^q a^q + y^q B^{q+q^2} C a \\
        &\quad + y^q B^{q^2} C^{q+1} a^{q+1} + y^q C^{1+q+q^2} + z A^{q+q^2} C^q + z A^q B^{q+q^2} \\
        &\quad + z A^q B^{q^2} C^q a^q + z A^{q^2} B^q C^q a + z B^{2q+q^2} a + z B^{q+q^2} C^q a^{q+1} \\
        &\quad + z^q A^{q+q^2} C + z^q A^q B^{q^2} C a^q + z^q B^q C^{q^2+1}.
    \end{align*}
    And $R_5$ factorizes (checked by hand and via Magma) into the product of
    \begin{align}
&A^{1+q+q^2} + A^{q+1} B^{q^2} a^q + A B^q C^{q^2} + A^q B^{q^2} C + A^{q^2} B C^q + A^{q^2} C^{q+1} a\nonumber \\ 
&\qquad + B^{1+q+q^2} + B^{q^2+1} C^q a^q + B^{q+q^2}  C a + B^{q^2} C^{q+1} a^{q+1} + C^{1+q+q^2}, \text{ and} \\
&\qquad y B^q C^q + z A^q C^q + z B^q C^q a + z^q A^q C + z^q B^{q+1} + z^q B^q C a.
   \end{align}
    Let
    \begin{align*}
    H &= A^{1+q+q^2} + A^{q+1} B^{q^2} a^q + A B^q C^{q^2} + A^q B^{q^2} C + A^{q^2} B C^q + A^{q^2} C^{q+1} a \\
      &\qquad + B^{1+q+q^2} + B^{q^2+1} C^q a^q + B^{q+q^2} C a + B^{q^2} C^{q+1} a^{q+1} + C^{1+q+q^2}.
    \end{align*}
    By direct Magma computations $$H=\prod_{\begin{array}{c}\lambda \in \overline{\mathbb{F}_2} :\\ \lambda^{q^2+q+1}+\lambda^{q^2}a^q=1\end{array}}\Big(\lambda^{q+1}A+\lambda^qB+(\lambda^qa+1)C\Big).$$
    By our initial assumption, there are no $\lambda\in \mathbb{F}_{2^m}$ satisfying $\lambda^{q^2+q+1}+\lambda^{q^2}a^q=1$ and there are not $(A,B,C)\in \mathbb{F}_{2^m}$, $ABC\neq 0$, making $H$ vanish. Thus $H\neq 0$.
    We can then divide $R_5$ by $H$ to obtain
    $$R_6:= y B^q C^q + z A^q C^q + z B^q C^q a + z^q A^q C + z^q B^{q+1} + z^q B^q C a = 0.$$
    Now we eliminate $y$ and $y^q$ from $R_3$ using $R_6$, obtaining the product of the factors $B^q$, $(A^{q^2} C^q + B^{q+q^2} + B^{q^2} C^q a^q)$, $(A^q C + B^{q+1} + B^q C a)$, $A$, and $(z^q C^{q^2} + z^{q^2} C^q)$.
    Since we are in the case where the parameters do not nullify the first factors, this forces
    $$ z^q C^{q^2} + z^{q^2} C^q = 0.$$
    By our assumptions on $q$, this yields only two solutions in $z$, and consequently two solutions in $(x,y,z)$.
\end{enumerate}

 Suppose now that  $T^{q^2+q+1}+T^{q^2+q}a+1$ has a solution in $\mathbb{F}_{2^m}$. Following the case $(A,B,C)=(A,0,C)$ with $A,C
\neq0$, we can easily find pairs $(A,C)\in \mathbb{F}_{2^m}^2$, $A,C\neq 0$, such that  $A^{q^2}C^{q+1}a+A^{q^2+q+1}+C^{q^2+q+1}=0$ and thus the above equation vanishes. Since $A^qx+Ax^q$ has rank $m-1$ and $A y^q + \frac{C^{2q+1}}{A^{q+1}} y + \frac{C^{q^2+q+1}}{A^{q^2+q}} y^q$ has rank at least $m-1$, there are at least $2^m$ pairs satisfying Equation~\eqref{Eq:Converse} and thus the function $H_a$ is not APN.

The theorem is shown.
\end{proof}

\begin{theorem}
\label{thm:no_root_iff_Permutation}
 Suppose  that $q=2^i$, $\gcd(i,m)=1$. Then $H_a$ is a permutation if and only if  $P_a'(T)=T^{q^2+q+1} + T^{q^2+q}a + 1 $ has no roots in $\mathbb{F}_{2^m}$.   
\end{theorem}

\begin{proof}
Let $H_a(\mathbf{v}) = (F_1(\mathbf{v}), F_2(\mathbf{v}), F_3(\mathbf{v}))$ where $\mathbf{v} = (x,y,z)$. To prove that $H_a$ is a permutation polynomial over $\mathbb{F}_{2^m}^3$, we must show that the polar derivative $D_{\mathbf{w}}H_a(\mathbf{v}) = H_a(\mathbf{v} + \mathbf{w}) + H_a(\mathbf{v}) = 0$ implies $\mathbf{w} = (0,0,0)$, where $\mathbf{w} = (\alpha, \beta, \gamma)$.

Since the characteristic is $2$, the condition $D_{\mathbf{w}}H_a(\mathbf{v}) = 0$ yields the following system of equations
\begin{align}
x\alpha^q + \alpha x^q + \alpha^{q+1} + a(x\beta^q + \alpha y^q + \alpha\beta^q) + y\gamma^q + \beta z^q + \beta\gamma^q &= 0 \label{eq:1} \\
x\beta^q + \alpha y^q + \alpha\beta^q + z\gamma^q + \gamma z^q + \gamma^{q+1} &= 0 \label{eq:2} \\
x^q\gamma + \alpha^q z + \alpha^q\gamma + y\beta^q + \beta y^q + \beta^{q+1} + a(y^q\gamma + \beta^q z + \beta^q\gamma) &= 0. \label{eq:3}
\end{align}

We proceed by analyzing two main cases based on the value of $\beta$.
\begin{enumerate}
    \item $\beta = 0$.\\
    Substituting $\beta = 0$ into equation \eqref{eq:2} yields:
$$z\gamma^q + \gamma z^q + \gamma^{q+1} = 0.$$
If $\gamma \neq 0$, dividing by $\gamma^{q+1}$ gives $(z/\gamma) + (z/\gamma)^q + 1 = 0$, which implies that the absolute trace $\operatorname{Tr}(1) = 0$. Since $m$ is odd and $\gcd(i,m)=1$, this is a contradiction. Thus, we must have $\gamma = 0$.

Substituting $\beta = 0$ and $\gamma = 0$ into equation \eqref{eq:1} provides:
$$x\alpha^q + \alpha x^q + \alpha^{q+1} = 0.$$
By the exact same trace argument over $\mathbb{F}_{2^m}$, this equation only admits the solution $\alpha = 0$. Therefore, if $\beta = 0$, the only possible solution is the trivial one, $\mathbf{w} = (0,0,0)$.

\item $\beta\neq 0 $ and $z\beta^{q^2} + \gamma y^{q^2} = 0$.\\
First, suppose $\gamma = 0$. Then $z\beta^{q^2} = 0$, which implies $z = 0$ since $\beta \neq 0$. Substituting $\gamma = 0$ and $z = 0$ into \eqref{eq:2} yields
$$x\beta^q + \alpha y^q + \alpha\beta^q = 0 \implies x = \frac{\alpha(y^q + \beta^q)}{\beta^q}.$$
Substituting these values into \eqref{eq:3} leaves:
$$y\beta^q + \beta y^q + \beta^{q+1} = 0.$$
Dividing by $\beta^{q+1}$ gives $(y/\beta)^q + (y/\beta) + 1 = 0$, which implies $\operatorname{Tr}(1) = 0$. Since $m$ is odd and $\gcd(i,m)=1$, this is a contradiction. Thus, we must have $\gamma \neq 0$.

Since both $\beta \neq 0$ and $\gamma \neq 0$, the condition $z\beta^{q^2} + \gamma y^{q^2} = 0$ can be rewritten as:
$$\frac{z}{\gamma} = \left(\frac{y}{\beta}\right)^{q^2}.$$
Let $u = y/\beta$. Then we can parameterize $y$ and $z$ as $y = u\beta$ and $z = u^{q^2}\gamma$. 

Substitute $y = u\beta$ and $z = u^{q^2}\gamma$ into equation \eqref{eq:2}:
$$x\beta^q =\frac{\alpha u^q\beta^q + \alpha\beta^q + u^{q^2}\gamma^{q+1} + u^{q^3}\gamma^{q+1} + \gamma^{q+1}}{\beta^q}.$$

Now substitute $y = u\beta$, $z = u^{q^2}\gamma$, and $x^q$ into equation \eqref{eq:3} one gets
$$\frac{\gamma^{q^2+q}}{\beta^{q^2}}(u^{q^3} + u^{q^4} + 1) + \frac{\beta^{q+1}}{\gamma}(u + u^q + 1) + a\beta^q(u^q + u^{q^2} + 1) = 0.$$

Let $W = u + u^q + 1$. Note that $W \neq 0$, because $W = 0$ would imply $\operatorname{Tr}(1) = 0$, which is impossible for odd $m$. Thus
$$\frac{\gamma^{q^2+q}}{\beta^{q^2}} W^{q^3} + \frac{\beta^{q+1}}{\gamma} W + a\beta^q W^q = 0\iff \frac{\gamma^{q^2+q+1}}{\beta^{q^2+q+1}} W^{q^3-1} + a \frac{\gamma}{\beta} W^{q-1} + 1 = 0.$$

Define a new variable $S = \frac{\gamma}{\beta} W^{q-1}$ and thus 
$$S^{q^2+q+1} + aS + 1 = 0.$$
This is exactly the polynomial $P_a(1/S) = 0$. If  $P_a$ has no roots in $\mathbb{F}_{2^m}$ then we got a contradiction. 

\item $\beta\neq 0 $ and $z\beta^{q^2} + \gamma y^{q^2} \neq 0$.\\
 By \eqref{eq:2} we deduce 
$$x=\frac{\alpha y^q + \alpha\beta^q + z\gamma^q + \gamma z^q + \gamma^{q+1}}{\beta^q}$$
and thus the remaining equation reads
\begin{align*}
R_1 &:= y\gamma^q\beta^{q+q^2} + az\gamma^q\beta^{q+q^2} + z\alpha^q\gamma^q\beta^{q^2} + \alpha^{q+1}y^q\beta^{q^2} + \alpha z^q\beta^q\gamma^{q^2} + \alpha^{q+1}y^{q^2}\beta^q \\
&\quad + \alpha^{q+1}\beta^{q+q^2} + \alpha z^{q^2}\beta^q\gamma^q + \alpha\beta^q\gamma^{q+q^2} + z^q\beta^{1+q+q^2} + \gamma^q\beta^{1+q+q^2} + a\gamma z^q\beta^{q+q^2} \\
&\quad + a\gamma^{q+1}\beta^{q+q^2} + \gamma z^q\alpha^q\beta^{q^2} + \alpha^q\gamma^{q+1}\beta^{q^2} = 0
\end{align*}
 and 
\begin{align*}
R_2 &:= y\beta^{q+q^2} + az\beta^{q+q^2} + z\alpha^q\beta^{q^2} + y^q\beta^{1+q^2} + \beta^{1+q+q^2} + a\gamma y^q\beta^{q^2} \\
&\quad + a\gamma\beta^{q+q^2} + \gamma\gamma^{q^2}z^q + \gamma\alpha^qy^{q^2} + \gamma^{q+1}z^{q^2} + \gamma^{1+q+q^2} = 0.
\end{align*}
From this last expression we deduce 
$$\alpha^q = \frac{y\beta^{q+q^2} + az\beta^{q+q^2} + y^q\beta^{1+q^2} + \beta^{1+q+q^2} + a\gamma y^q\beta^{q^2} + a\gamma\beta^{q+q^2} + \gamma\gamma^{q^2}z^q + \gamma^{q+1}z^{q^2} + \gamma^{1+q+q^2}}{z\beta^{q^2} + \gamma y^{q^2}},$$
and combining with the other equation (after dividing by a suitable power of $\beta$) we obtain (we used Magma to simplify)
{\small
\begin{align*}
R_3 &:= y^{1+q+q^2}\beta^{q+q^2+q^3} + y^{1+q+q^3}\beta^{q+2q^2} + y^{q+1}\beta^{q+2q^2+q^3} + a^qy^{1+q^2}z^q\beta^{q+q^2+q^3} + a^qy^{1+q^3}z^q\beta^{q+2q^2} \\
&\quad + a^qyz^q\beta^{q+2q^2+q^3} + yz^{q+q^2}\beta^{q+q^2}\gamma^{q^3} + yz^{q+q^3}\beta^{q+q^2}\gamma^{q^2} + yz^q\beta^{q+q^2}\gamma^{q^2+q^3} + y^{1+2q^2}\beta^{2q+q^3} \\
&\quad + y^{1+q^2+q^3}\beta^{2q+q^2} + y\beta^{2q+2q^2+q^3} + y\beta^{2q+2q^2+q^3} + a^qy^{1+2q^2}\beta^{q+q^3}\gamma^q + a^qy^{1+q^2+q^3}\beta^{q+q^2}\gamma^q \\
&\quad + a^qy^{1+q^3}\beta^{q+2q^2}\gamma^q + a^qy\beta^{q+2q^2+q^3}\gamma^q + y^{1+q^2}z^{q^2}\beta^q\gamma^{q+q^3} + y^{1+q^2}z^{q^3}\beta^q\gamma^{q+q^2} + y^{1+q^2}\beta^q\gamma^{q+q^2+q^3} \\
&\quad + yz^{q^2}\beta^{q+q^2}\gamma^{q+q^3} + yz^{q^3}\beta^{q+q^2}\gamma^{q+q^2} + y\beta^{q+q^2}\gamma^{q+q^2+q^3} + a y^{q+q^2}z\beta^{q+q^2+q^3} + a y^{q+q^3}z\beta^{q+2q^2} \\
&\quad + a y^qz\beta^{q+2q^2+q^3} + a a^qy^{q^2}z^{1+q}\beta^{q+q^2+q^3} + a a^qy^{q^3}z^{1+q}\beta^{q+2q^2} + a a^qz^{1+q}\beta^{q+2q^2+q^3} + a z^{1+q+q^2}\beta^{q+q^2}\gamma^{q^3} \\
&\quad + a z^{1+q+q^3}\beta^{q+q^2}\gamma^{q^2} + a z^{1+q}\beta^{q+q^2}\gamma^{q^2+q^3} + a y^{2q^2}z\beta^{2q+q^3} + a y^{q^2+q^3}z\beta^{2q+q^2} + a y^{q^3}z\beta^{2q+2q^2} \\
&\quad + a z\beta^{2q+2q^2+q^3} + a a^qy^{2q^2}z\beta^{q+q^3}\gamma^q + a a^qy^{q^2+q^3}z\beta^{q+q^2}\gamma^q + a a^qy^{q^3}z\beta^{q+2q^2}\gamma^q + a a^qz\beta^{q+2q^2+q^3}\gamma^q \\
&\quad + a y^{q^2}z^{1+q^2}\beta^q\gamma^{q+q^3} + a y^{q^2}z^{1+q^3}\beta^q\gamma^{q+q^2} + a y^{q^2}z\beta^q\gamma^{q+q^2+q^3} + a z^{1+q^2}\beta^{q+q^2}\gamma^{q+q^3} + a z^{1+q^3}\beta^{q+q^2}\gamma^{q+q^2} \\
&\quad + a z\beta^{q+q^2}\gamma^{q+q^2+q^3} + y^qz^{1+q^2}\beta^{q^2+q^3}\gamma^q + y^{q+q^3}z\beta^{q^2}\gamma^{q+q^2} + y^qz\beta^{q^2+q^3}\gamma^{q+q^2} + y^{q^2}z^{1+q}\beta^{q+q^3}\gamma^{q^2} \\
&\quad + z^{1+q+q^2}\beta^{q+q^2+q^3} + a^qy^{q^2}z^{1+q}\beta^{q^3}\gamma^{q+q^2} + a^qy^{q^3}z^{1+q}\beta^{q^2}\gamma^{q+q^2} + a^qz^{1+q}\beta^{q^2+q^3}\gamma^{q+q^2} + z^{1+q}\gamma^{q+q^2+q^3} \\
&\quad + z^{1+q+q^3}\gamma^{q+2q^2} + z^{1+q}\gamma^{q+2q^2+q^3} + y^{q^2}z^{1+q^2}\beta^{q+q^3}\gamma^q + y^{q^2}z\beta^{q^3}\gamma^{q+q^2} + y^{q^3}z^{1+q^2}\beta^{q^2}\gamma^q \\
&\quad + z^{1+q^2}\beta^{q^2+q^3}\gamma^q + y^{q^3}z\beta^{q^2}\gamma^{q+q^2} + z\beta^{q^2+q^3}\gamma^{q+q^2} + a^qy^{q^2}z^{1+q^2}\beta^{q^3}\gamma^{2q} + a^qy^{q^2}z\beta^{q^3}\gamma^{2q+q^2} \\
&\quad + a^qy^{q^3}z^{1+q^2}\beta^{q^2}\gamma^{2q} + a^qz^{1+q^2}\beta^{q^2+q^3}\gamma^{2q} + a^qy^{q^3}z\beta^{q^2}\gamma^{2q+q^2} + a^qz\beta^{q^2+q^3}\gamma^{2q+q^2} + z^{1+2q^2}\gamma^{2q+q^3} \\
&\quad + z^{1+q^2+q^3}\gamma^{2q+q^2} + z\gamma^{2q+2q^2+q^3} + z\gamma^{2q+2q^2+q^3} + y^{q^2+2q}\beta^{1+q^2+q^3} + y^{2q+q^3}\beta^{1+2q^2} \\
&\quad + y^{2q}\beta^{1+2q^2+q^3} + a^qy^{q+q^2}z^q\beta^{1+q^2+q^3} + a^qy^{q+q^3}z^q\beta^{1+2q^2} + a^qy^qz^q\beta^{1+2q^2+q^3} + y^qz^{q+q^2}\beta^{1+q^2}\gamma^{q^3} \\
&\quad + y^qz^{q+q^3}\beta^{1+q^2}\gamma^{q^2} + y^qz^q\beta^{1+q^2}\gamma^{q^2+q^3} + y^{q+2q^2}\beta^{1+q+q^3} + y^{q+q^2+q^3}\beta^{1+q+q^2} + y^q\beta^{1+q+q^2+q^3} \\
&\quad + a^qy^{q+2q^2}\beta^{1+q^3}\gamma^q + a^qy^{q+q^2+q^3}\beta^{1+q^2}\gamma^q + a^qy^{q+q^3}\beta^{1+2q^2}\gamma^q + a^qy^q\beta^{1+2q^2+q^3}\gamma^q + y^qz^{q^2}\beta^{1+q}\gamma^{q+q^3} \\
&\quad + y^qz^{q^3}\beta^{1+q}\gamma^{q+q^2} + y^q\beta^{1+q}\gamma^{q+q^2+q^3} + y^qz^{q^2}\beta^{1+q^2}\gamma^{q+q^3} + z^{q^3}\beta^{1+q^2}\gamma^{q+q^2} + \beta^{1+q^2}\gamma^{q+q^2+q^3} \\
&\quad + a^qy^{q^2}z^{2q}\beta^{1+q^2+q^3} + a^qy^{q^3}z^{2q}\beta^{1+2q^2} + a^qz^{2q}\beta^{1+2q^2+q^3} + z^{2q+q^2}\beta^{1+q^2}\gamma^{q^3} + z^{2q+q^3}\beta^{1+q^2}\gamma^{q^2} \\
&\quad + z^{2q}\beta^{1+q^2}\gamma^{q^2+q^3} + y^{2q^2}\beta^{1+2q+q^3} + y^{q^2+q^3}\beta^{1+2q+q^2} + y^{q^3}\beta^{1+2q+2q^2} + \beta^{1+2q+2q^2+q^3} \\
&\quad + a^qy^{2q^2}\beta^{1+q+q^3}\gamma^q + a^qy^{q^2+q^3}\beta^{1+q+q^2}\gamma^q + a^qy^{q^3}\beta^{1+q+2q^2}\gamma^q + a^q\beta^{1+q+2q^2+q^3}\gamma^q + y^{q^2}z^{q^2}\beta^{1+q}\gamma^{q+q^3} \\
&\quad + y^{q^2}z^{q^3}\beta^{1+q}\gamma^{q+q^2} + y^{q^2}\beta^{1+q}\gamma^{q+q^2+q^3} + z^{q^2}\beta^{1+q+q^2}\gamma^{q+q^3} + z^{q^3}\beta^{1+q+q^2}\gamma^{q+q^2} + \beta^{1+q+q^2}\gamma^{q+q^2+q^3} \\
&\quad + a y^{q^2+2q}\gamma\beta^{q^2+q^3} + a y^{2q+q^3}\gamma\beta^{2q^2} + a y^{2q}\gamma\beta^{2q^2+q^3} + a a^qy^{q+q^2}z^q\gamma\beta^{q^2+q^3} + a a^qy^{q+q^3}z^q\gamma\beta^{2q^2} \\
&\quad + a a^qy^qz^q\gamma\beta^{2q^2+q^3} + a y^qz^{q+q^2}\gamma\beta^{q^2}\gamma^{q^3} + a y^qz^{q+q^3}\gamma\beta^{q^2}\gamma^{q^2} + a y^qz^q\gamma\beta^{q^2}\gamma^{q^2+q^3} + a y^{q+2q^2}\gamma\beta^{q+q^3} \\
&\quad + a y^{q+q^2+q^3}\gamma\beta^{q+q^2} + a y^q\gamma\beta^{q+q^2+q^3} + a a^qy^{q+2q^2}\gamma^{1+q}\beta^{q^3} + a a^qy^{q+q^2+q^3}\gamma^{1+q}\beta^{q^2} + a a^qy^{q+q^3}\gamma^{1+q}\beta^{2q^2} \\
&\quad + a a^qy^q\gamma^{1+q}\beta^{2q^2+q^3} + a y^qz^{q^2}\gamma^{1+q+q^3} + a y^qz^{q^3}\gamma^{1+q+q^2} + a y^q\gamma^{1+q+q^2+q^3} + a y^qz^{q^2}\gamma^{1+q}\beta^{q^2}\gamma^{q^3} \\
&\quad + a y^q\gamma^{1+q}\beta^{q^2}z^{q^3} + a y^q\gamma^{1+q}\beta^{q^2}\gamma^{q^2+q^3} + a a^qy^{q^2}z^{2q}\gamma\beta^{q^2+q^3} + a a^qy^{q^3}z^{2q}\gamma\beta^{2q^2} + a a^qz^{2q}\gamma\beta^{2q^2+q^3} \\
&\quad + a z^{2q+q^2}\gamma\beta^{q^2}\gamma^{q^3} + a z^{2q+q^3}\gamma\beta^{q^2}\gamma^{q^2} + a z^{2q}\gamma\beta^{q^2}\gamma^{q^2+q^3} + a y^{2q^2}\gamma\beta^{2q+q^3} + a y^{q^2+q^3}\gamma\beta^{2q+q^2} \\
&\quad + a y^{q^3}\gamma\beta^{2q+2q^2} + a\gamma\beta^{2q+2q^2+q^3} + a a^qy^{2q^2}\gamma^{1+q}\beta^{q+q^3} + a a^qy^{q^2+q^3}\gamma^{1+q}\beta^{q+q^2} + a a^qy^{q^3}\gamma^{1+q}\beta^{q+2q^2} \\
&\quad + a a^q\gamma^{1+q}\beta^{q+2q^2+q^3} + a y^{q^2}z^{q^2}\gamma^{1+q+q^3} + a y^{q^2}z^{q^3}\gamma^{1+q+q^2} + a y^{q^2}\gamma^{1+q+q^2+q^3} + a z^{q^2}\gamma^{1+q}\beta^{q^2}\gamma^{q^3} \\
&\quad + a \gamma^{1+q}\beta^{q^2}z^{q^3} + a \gamma^{1+q}\beta^{q^2}\gamma^{q^2+q^3} + y^{q+q^2}z^q\gamma^{1+q^2}\beta^{q^3} + y^{q+q^3}z^q\gamma^{1+q^2}\beta^{q^2} + y^qz^q\gamma^{1+q^2}\beta^{q^2+q^3} \\
&\quad + y^{q+q^2}\gamma^{1+q}\gamma^{q^2}y^{q^3} + y^qz^{q^2}\gamma^{1+q}\beta^{q^2}y^{q^3} + y^qz^{q^2}\gamma^{1+q}\beta^{q^2+q^3} + y^q\gamma^{1+q}\beta^{q^2}\gamma^{q^2}y^{q^3} + y^q\gamma^{1+q}\beta^{q^2+q^3}\gamma^{q^2} \\
&\quad + a^qy^{q^2}z^{2q}\gamma^{1+q^2}\beta^{q^3} + a^qy^{q^3}z^{2q}\gamma^{1+q^2}\beta^{q^2} + a^qz^{2q}\gamma^{1+q^2}\beta^{q^2+q^3} + z^{q^2+2q}\gamma^{1+q^2+q^3} + z^{2q+q^3}\gamma^{1+q^2+q^2} \\
&\quad + z^{2q}\gamma^{1+q^2+q^2+q^3} + y^{q^2}z^q\beta^q\gamma^{1+q^2}\beta^{q^3} + y^{q^2}z^q\beta^q\gamma^{1+q^2}y^{q^3} + y^qz^q\beta^{q+q^2}\gamma^{1+q^2}y^{q^3} + z^q\beta^{q+q^2}\gamma^{1+q^2+q^3} \\
&\quad + a^qy^{q^2}z^{q+q^2}\gamma^{1+q}\beta^{q^3} + a^qy^{q^3}z^{q+q^2}\gamma^{1+q}\beta^{q^2} + a^qz^{q+q^2}\gamma^{1+q}\beta^{q^2+q^3} + z^{2q^2+q}\gamma^{1+q+q^3} + z^{q+q^2+q^3}\gamma^{1+q+q^2} \\
&\quad + z^q\gamma^{1+q+2q^2+q^3} + y^{q^2}z^{q^2}\beta^q\gamma^{1+q}\beta^{q^3} + y^{q^2}\beta^q\gamma^{1+q+q^2}\beta^{q^3} + z^{q^2}\beta^{q+q^2}\gamma^{1+q}y^{q^3} + z^{q^2}\beta^{q+q^2+q^3}\gamma^{1+q} \\
&\quad + \beta^{q+q^2}\gamma^{1+q+q^2}y^{q^3} + \beta^{q+q^2+q^3}\gamma^{1+q+q^2} + a^qy^{q^2}z^{q^2}\gamma^{1+2q}\beta^{q^3} + a^qy^{q^2}\gamma^{1+2q+q^2}\beta^{q^3} + a^qz^{q^2}\beta^{q^2}\gamma^{1+2q}y^{q^3} \\
&\quad + a^qz^{q^2}\beta^{q^2+q^3}\gamma^{1+2q} + a^q\beta^{q^2}\gamma^{1+2q+q^2}y^{q^3} + a^q\beta^{q^2+q^3}\gamma^{1+2q+q^2} + z^{2q^2}\gamma^{1+2q+q^3} + z^{q^2+q^3}\gamma^{1+2q+q^2} \\
&\quad + z^{q^3}\gamma^{1+2q+2q^2} + \gamma^{1+2q+2q^2+q^3} = 0.
\end{align*}
}

Now, 
\begin{equation}\label{R_3_prod}
R_3 = \prod_{\lambda\in \Lambda} \Big( y^q\beta + y\beta^q + \beta^{q+1} + \lambda(y^q\gamma + \gamma\beta^q + z\beta^q) + (a\lambda^q + \lambda^{q+1})(z\gamma^q + \gamma^{q+1} + z^q\gamma) \Big),
\end{equation}
where 
$$\Lambda := \left\{\lambda \in \overline{\mathbb{F}_2} : a^{q+1}\lambda^{q^2} + a\lambda^{q+q^2} + a^q\lambda^{1+q^2} + \lambda^{1+q+q^2} + 1 = 0 \right\}.$$

Note that $\lambda\in \mathbb{F}_{2^m}$ if and only if 
$$\left(\frac{1}{\lambda}\right)^{q^2+q+1} + \left(\frac{a}{\lambda}+1\right)^{q+1} = 0,$$
that is, $P_a(1/\lambda) = 0$. Thus, all the factors in \eqref{R_3_prod} are defined over $\mathbb{F}_{2^m}$ if and only if $P_a(T)$ has roots in $\mathbb{F}_{2^m}$. Also, each factor $y^q\beta + y\beta^q + \beta^{q+1} + \lambda(y^q\gamma + \gamma\beta^q + z\beta^q) + (a\lambda^q + \lambda^{q+1})(z\gamma^q + \gamma^{q+1} + z^q\gamma)$ is absolutely irreducible. We distinguish two cases:

\begin{enumerate}
    \item $a \neq \lambda$. The partial derivatives with respect to $y, z, \beta, \gamma$ are 
    $$\begin{cases}
    \beta^q = 0 \\
    \lambda \beta^q + (a\lambda^q + \lambda^{q+1})\gamma^q = 0 \\
    y^q + \beta^q = 0 \\
    \lambda(y^q + \beta^q) + (a\lambda^q + \lambda^{q+1})(\gamma^q + z^q) = 0.
    \end{cases}$$
    Since $\lambda \neq 0$, the only possible solution to the above system is $(0,0,0,0)$. Thus, the hypersurface in $\mathbb{P}^3(\overline{\mathbb{F}_2})$ defined by the above equation is non-singular and therefore absolutely irreducible. 
    
    \item $a = \lambda$. Then the factor reads $y^q\beta + y\beta^q + \beta^{q+1} + \lambda(y^q\gamma + \gamma\beta^q + z\beta^q)$, which is clearly absolutely irreducible since it is of degree $1$ in $z$ and $\gcd(y^q\beta + y\beta^q + \beta^{q+1} + \lambda(y^q\gamma + \gamma\beta^q), \lambda \beta^q) = 1$. 
\end{enumerate}

Suppose that there exists a $\lambda \in \Lambda\cap \mathbb{F}_{2^m}$. Then the corresponding factor $y^q\beta + y\beta^q + \beta^{q+1} + \lambda(y^q\gamma + \gamma\beta^q + z\beta^q) + (a\lambda^q + \lambda^{q+1})(z\gamma^q + \gamma^{q+1} + z^q\gamma)$ is absolutely irreducible, and if $m$ is large enough, there exist $(\alpha, \beta, \gamma, x, y, z) \in \mathbb{F}_{2^m}$ with $\alpha\beta\gamma \neq 0$ satisfying the three equations \eqref{eq:1}, \eqref{eq:2}, and \eqref{eq:3}, meaning $H_a$ is not a permutation. 

Suppose now that $\Lambda\cap \mathbb{F}_{2^m} = \emptyset$ (that is, $T^{q^2+q+1} + aT^{q^2+q} + 1$ has no roots in $\mathbb{F}_{2^m}$). First, note that $\{1, \lambda, a\lambda^q + \lambda^{q+1}\}$ are linearly independent over $\mathbb{F}_{2^m}$. Suppose on the contrary that 
$a\lambda^q + \lambda^{q+1} = A\lambda + B,$
with $A, B \in \mathbb{F}_{2^m}$. Then
$A^{q+1} \lambda + BA^q + a B^q + \lambda B^q + 1 = 0.$
Since $\lambda \notin \mathbb{F}_{2^m}$, we must have $A^{q+1} = B^q$ and $BA^q + a B^q + 1 = 0$, yielding 
$\sqrt[q]{A^{q+1}} A^q + a A^{q+1} + 1 = 0,$
and so $\sqrt[q]{A}$ is a solution to $T^{q^2+q+1} + aT^{q^2+q} + 1 = 0$, which is a contradiction. 
Thus, the $\mathbb{F}_{2^m}$-solutions of $y^q\beta + y\beta^q + \beta^{q+1} + \lambda(y^q\gamma + \gamma\beta^q + z\beta^q) + (a\lambda^q + \lambda^{q+1})(z\gamma^q + \gamma^{q+1} + z^q\gamma) = 0$ must satisfy 
$$y^q\beta + y\beta^q + \beta^{q+1} = y^q\gamma + \gamma\beta^q + z\beta^q = z\gamma^q + \gamma^{q+1} + z^q\gamma = 0,$$
which is a contradiction to $\beta \neq 0$, since $y^q\beta + y\beta^q + \beta^{q+1} = 0$ yields $\beta = 0$ because $m$ is odd. Therefore, in this case, the system \eqref{eq:1}, \eqref{eq:2}, \eqref{eq:3} has no non-trivial solutions in $\mathbb{F}_{2^m}$, and $H_a$ is a permutation.
\end{enumerate}

The theorem is shown.
\end{proof}

\begin{corollary}
\label{thm:Ha}
Let $m\geq 3$ be odd, $\gcd(i,m)=1$, $q=2^i$, $a\in\F_{2^m}^*$ and $H_a(x,y,z) = \bigl(x^{q+1}+axy^q+yz^q,\; xy^q+z^{q+1},\; x^qz+y^{q+1}+ay^qz\bigr)$ on $\F_{2^{3m}}$.
The following are equivalent:
\begin{enumerate}
\item[\textup{(i)}] $R_a(T) = T^{q^2+q+1}+(aT+1)^{q+1} \in \F_{2^m}[T]$ has no root in $\F_{2^m}$.
\item[\textup{(ii)}] $H_a$ is a permutation on $\F_{2^m}^3$.
\item[\textup{(iii)}] $H_a$ is APN on $\F_{2^m}^3$.
\end{enumerate}
Moreover, $R_a$ is root-equivalent to $Q_a$ (Proposition~\textup{\ref{prop:poly_equivalences}}), so the same values of $a$ that make $G_a$ an APN permutation also make $H_a$ an APN permutation.
\end{corollary}

\begin{corollary}\label{cor:Ha_existence}
Assume $q=2^i$ with $\gcd(i,m)=1$ and $m\geq 3$ odd. Then there exist at least
$\frac{2^m + 1 - (d-1)(d-2)2^{m/2} - d}{d}$ elements in $\mathbb{F}_{2^m}$ such that both $G_a$ and $H_a$ are APN.
\end{corollary}

\begin{proof}
By Theorem~\ref{thm:existence}, there exist at least $\frac{2^m + 1 - (d-1)(d-2)2^{m/2} - d}{d}$   number of $a\in\F_{2^m}^*$ such that $G_a$ is
APN. By the root-equivalence established earlier for the families $G_a$ and
$H_a$, the same parameter $a$ also gives that $H_a$ is APN.
\end{proof}

\section{Diagonal equivalence to the Li--Kaleyski representatives}
\label{sec:diag_equiv}

We continue by asking when the generalized families $G_a$ and $H_a$ are
equivalent to the Li--Kaleyski representatives
$F_1=G_1$ and  $F_2=H_1$.
Recall that two functions $F,G\colon \F_{2^m}^3\to \F_{2^m}^3$ are
\emph{affinely equivalent} if there exist affine bijections
$A_1(\mathbf{u})=L_1(\mathbf{u})+c_1$ and
$A_2(\mathbf{x})=L_2(\mathbf{x})+c_2$ such that
$G=A_1\circ F\circ A_2$. If $c_1=c_2=0$, this reduces to
\emph{linear equivalence}. In the present section we do not attempt to solve
the full affine-equivalence problem for the families $G_a$ and $H_a$.
Instead, we restrict to $\F_{2^m}$-linear maps and, within that class, to the
diagonal subclass. This is natural because both families are built from
coordinatewise monomials of the form $x_i^q x_j$, so diagonal scalings are the
first symmetries to test; moreover, this already suffices to show that many
good parameters yield functions inequivalent to the corresponding
Li--Kaleyski representatives.

Accordingly, throughout this section we consider diagonal bijections
\[
A_1(u,v,w)=(\mu u,\nu v,\rho w),
\qquad
A_2(x,y,z)=(\lambda_1x,\lambda_2y,\lambda_3z),
\]
with $\mu,\nu,\rho,\lambda_1,\lambda_2,\lambda_3\in\F_{2^m}^*$.

\begin{remark}[Affine versus linear equivalence]
\label{rem:affine_linear_scope}
Two functions $F,G\colon \F_{2^m}^3\to \F_{2^m}^3$ are \emph{affinely equivalent} if 
$G = A_1\circ F\circ A_2$ for affine bijections $A_j(\mathbf{x}) = L_j(\mathbf{x}) + \mathbf{c}_j$.
Since $G_a(\mathbf{0})=H_a(\mathbf{0})=\mathbf{0}$ for all $a$, any affine equivalence 
$G_1 = A_1\circ G_a\circ A_2$ imposes constraints on the translation parts $\mathbf{c}_1, \mathbf{c}_2$.
We show here that these translations must vanish, reducing affine equivalence to linear equivalence.

Indeed, $G_1$ is homogeneous of degree $q+1$ (algebraic degree 2 over $\F_2$), 
so it contains no linear or degree-$q$ terms. 
Consider the expansion of $G_a(N\mathbf{x}+\mathbf{c}_2)$. 
Terms of the form $(x_i+c_{2i})^{q+1}$ expand to $x_i^{q+1} + x_i^q c_{2i} + x_i c_{2i}^q + c_{2i}^{q+1}$.
The presence of $x_i^q c_{2i}$ (degree $q$) and $x_i c_{2i}^q$ (degree 1) implies that 
if $\mathbf{c}_2 \neq \mathbf{0}$, the composition $A_1(G_a(A_2(\mathbf{x})))$ would contain 
linear and degree-$q$ terms that cannot be cancelled by $G_1(\mathbf{x})$.

Explicitly, collecting the coefficients of the linear terms in $\mathbf{y}=N\mathbf{x}$ 
from the expansion of $G_a(\mathbf{y}+\mathbf{c}_2)$ yields the system
\[
\begin{pmatrix}
c_{21}^q & c_{23}^q & a c_{21}^q \\
0 & c_{22}^q & c_{21}^q \\
c_{22}^q & 0 & a c_{22}^q+c_{23}^q
\end{pmatrix}
\mathbf{y}=
\begin{pmatrix}
0 \\
0 \\
0
\end{pmatrix}.
\]
For this to hold for all $\mathbf{y}$, we must have $c_{21}=c_{22}=c_{23}=0$, i.e., $\mathbf{c}_2=\mathbf{0}$.
Substituting $\mathbf{c}_2=\mathbf{0}$ into the constant terms gives 
$L_1(G_a(\mathbf{0})) + \mathbf{c}_1 = \mathbf{0}$. Since $G_a(\mathbf{0})=\mathbf{0}$, this forces $\mathbf{c}_1=\mathbf{0}$.
An analogous argument holds for the family $H_a$. 
Thus, we restrict our attention to \emph{linear} equivalences ($A_1, A_2$ linear) and, 
within that class, to the diagonal subclass.
\end{remark}

The main point is that both families lead to the same numerical condition on
the parameter~$a$.

\begin{theorem}
\label{thm:diag_equiv_both}
For $a\in\F_{2^m}^*$, the following are equivalent:
\begin{itemize}
    \item[\textup{(i)}] $G_1=A_1\circ G_a\circ A_2$ for some diagonal $\F_{2^m}$-linear bijections $A_1,A_2$;
    \item[\textup{(ii)}] $H_1=A_1\circ H_a\circ A_2$ for some diagonal $\F_{2^m}$-linear bijections $A_1,A_2$;
    \item[\textup{(iii)}] $a^{q^2+q+1}=1$.
\end{itemize}
Hence the set of parameters $a\in\F_{2^m}^*$ for which either $G_a$ is
diagonally equivalent to $G_1$ or $H_a$ is diagonally equivalent to $H_1$ is
precisely the subgroup
\[
\{a\in\F_{2^m}^*: a^{q^2+q+1}=1\},
\]
which has order
\[
d_0:=\gcd(q^2+q+1,\,2^m-1).
\]
Consequently, whenever $|\mathcal{B}_{m,q}|>d_0$, there exist good
parameters $a$ such that both $G_a$ and $H_a$ are APN permutations, but
neither is diagonally equivalent to its Li--Kaleyski representative.
\end{theorem}

\begin{proof}
We treat the two families separately.

\smallskip
\noindent
\textbf{Step 1: the family \texorpdfstring{$G_a$}{Ga}.}
Recall that
$G_a(x,y,z)=\bigl(x^{q+1}+ax^qz+yz^q,\;x^qz+y^{q+1},\;xy^q+ay^qz+z^{q+1}\bigr)$.
Computing $G_a(\lambda_1x,\lambda_2y,\lambda_3z)$ gives
\[
\bigl(
\lambda_1^{q+1}x^{q+1}+a\lambda_1^q\lambda_3\,x^qz+\lambda_2\lambda_3^q\,yz^q,\;
\lambda_1^q\lambda_3\,x^qz+\lambda_2^{q+1}y^{q+1},\;
\lambda_1\lambda_2^q\,xy^q+a\lambda_2^q\lambda_3\,y^qz+\lambda_3^{q+1}z^{q+1}
\bigr).
\]
Applying $A_1(u,v,w)=(\mu u,\nu v,\rho w)$ and comparing with
$G_1(x,y,z)=\bigl(x^{q+1}+x^qz+yz^q,\;x^qz+y^{q+1},\;xy^q+y^qz+z^{q+1}\bigr)$
yields the coefficient system
\begin{alignat}{2}
\mu\lambda_1^{q+1} &= 1, &\qquad \mu a\lambda_1^q\lambda_3 &= 1,
\tag{E1--E2}\label{eq:E1E2_complete}\\
\mu\lambda_2\lambda_3^q &= 1, &\qquad \nu\lambda_1^q\lambda_3 &= 1,
\tag{E3--E4}\label{eq:E3E4_complete}\\
\nu\lambda_2^{q+1} &= 1, &\qquad \rho\lambda_1\lambda_2^q &= 1,
\tag{E5--E6}\label{eq:E5E6_complete}\\
\rho a\lambda_2^q\lambda_3 &= 1, &\qquad \rho\lambda_3^{q+1} &= 1.
\tag{E7--E8}\label{eq:E7E8_complete}
\end{alignat}

We first prove necessity. From (E1), (E5), and (E8) we obtain
\[
\mu=\lambda_1^{-(q+1)},\qquad
\nu=\lambda_2^{-(q+1)},\qquad
\rho=\lambda_3^{-(q+1)}.
\]
Dividing (E2) by (E1) gives
\[
\frac{\mu a\lambda_1^q\lambda_3}{\mu\lambda_1^{q+1}}=1,
\qquad\text{hence}\qquad
\lambda_3=\lambda_1/a.
\]
Likewise, dividing (E7) by (E8) gives
\[
\frac{\rho a\lambda_2^q\lambda_3}{\rho\lambda_3^{q+1}}=1,
\qquad\text{hence}\qquad
\lambda_3^q=a\lambda_2^q.
\]
Now set $r:=\lambda_1/\lambda_2$. Since $\lambda_3=\lambda_1/a$, the relation
$\lambda_3^q=a\lambda_2^q$ becomes
$(\lambda_1/a)^q=a\lambda_2^q$,
so
$r^q=a^{q+1}$.
On the other hand, dividing (E2) by (E4) gives
$\frac{\mu a\lambda_1^q\lambda_3}{\nu\lambda_1^q\lambda_3}=1$,
 hence 
$a\mu=\nu$.
Substituting the expressions for $\mu$ and $\nu$ yields
\[
a\lambda_1^{-(q+1)}=\lambda_2^{-(q+1)},
\qquad\text{equivalently}\qquad
\lambda_1^{q+1}=a\lambda_2^{q+1},
\]
that is, $r^{q+1}=a$.
Dividing this last identity by $r^q=a^{q+1}$ gives $r=a^{-q}$.
Raising to the $q$-th power, we obtain $r^q=a^{-q^2}$. Since also
$r^q=a^{q+1}$, it follows that
$a^{-q^2}=a^{q+1}$,
and therefore $a^{q^2+q+1}=1$.

We now prove sufficiency. Assume that $a^{q^2+q+1}=1$. Choose any
$\lambda_2\in\F_{2^m}^*$, and define
\[
\lambda_1:=a^{-q}\lambda_2,
\qquad
\lambda_3:=\lambda_1/a=a^{-q-1}\lambda_2.
\]
Then
$\lambda_3^q=a^{-q^2-q}\lambda_2^q$.
Because $a^{q^2+q+1}=1$, we have $a^{-(q^2+q)}=a$, and hence
$\lambda_3^q=a\lambda_2^q$.
Moreover, if $r:=\lambda_1/\lambda_2=a^{-q}$, then $r^{q+1}=a^{-q(q+1)}=a^{-(q^2+q)}=a$.
Thus the relations $\lambda_3=\lambda_1/a$, $\lambda_3^q=a\lambda_2^q$, and
$\lambda_1^{q+1}=a\lambda_2^{q+1}$ all hold. Now define
\[
\mu:=\lambda_1^{-(q+1)},\qquad
\nu:=\lambda_2^{-(q+1)},\qquad
\rho:=\lambda_3^{-(q+1)}.
\]
Then (E1), (E5), and (E8) hold by definition. Also,
$\mu a\lambda_1^q\lambda_3
=
\lambda_1^{-(q+1)}a\lambda_1^q(\lambda_1/a)=1$,
so (E2) holds; similarly,
$\mu\lambda_2\lambda_3^q
=
\lambda_1^{-(q+1)}\lambda_2(a\lambda_2^q)
=
a\lambda_2^{q+1}\lambda_1^{-(q+1)}
=1$,
because $\lambda_1^{q+1}=a\lambda_2^{q+1}$, so (E3) holds. Next,
$\nu\lambda_1^q\lambda_3
=
\lambda_2^{-(q+1)}\lambda_1^q(\lambda_1/a)
=
\lambda_2^{-(q+1)}\lambda_1^{q+1}/a
=1$,
again because $\lambda_1^{q+1}=a\lambda_2^{q+1}$, so (E4) holds. Finally,
$\rho\lambda_1\lambda_2^q
=
\lambda_3^{-(q+1)}\lambda_1\lambda_2^q
=
(\lambda_1/a)^{-(q+1)}\lambda_1\lambda_2^q
=
a^{q+1}\lambda_1^{-q}\lambda_2^q
=1$,
since $\lambda_1^q=a^{q+1}\lambda_2^q$, and
$\rho a\lambda_2^q\lambda_3
=
\lambda_3^{-(q+1)}a\lambda_2^q\lambda_3
=
a\lambda_2^q\lambda_3^{-q}
=
a\lambda_2^q/(a\lambda_2^q)
=1$,
so (E6) and (E7) hold as well. Hence
$G_1=A_1\circ G_a\circ A_2$.
We conclude that $G_1$ is diagonally equivalent to $G_a$ if and only if
$a^{q^2+q+1}=1$.

\smallskip
\noindent
\textbf{Step 2: the family \texorpdfstring{$H_a$}{Ha}.}
Recall that
$H_a(x,y,z)=\bigl(x^{q+1}+axy^q+yz^q,\;xy^q+z^{q+1},\;x^qz+y^{q+1}+ay^qz\bigr)$.
Computing $H_a(\lambda_1x,\lambda_2y,\lambda_3z)$ gives
\begin{align*}
\bigl(
\lambda_1^{q+1}x^{q+1}+a\lambda_1\lambda_2^q\,xy^q+\lambda_2\lambda_3^q\,yz^q,\;
\lambda_1\lambda_2^q\,xy^q+\lambda_3^{q+1}z^{q+1},\; \lambda_1^q\lambda_3\,x^qz+\lambda_2^{q+1}y^{q+1}+a\lambda_2^q\lambda_3\,y^qz
\bigr).
\end{align*}
Applying $A_1$ and comparing with
$H_1(x,y,z)=\bigl(x^{q+1}+xy^q+yz^q,\;xy^q+z^{q+1},\;x^qz+y^{q+1}+y^qz\bigr)$
gives the system
\begin{alignat}{2}
\mu\lambda_1^{q+1}&=1,
&\qquad
\mu a\lambda_1\lambda_2^q&=1,
\tag{H1--H2}\label{eq:H1H2_complete}\\
\mu\lambda_2\lambda_3^q&=1,
&\qquad
\nu\lambda_1\lambda_2^q&=1,
\tag{H3--H4}\label{eq:H3H4_complete}\\
\nu\lambda_3^{q+1}&=1,
&\qquad
\rho\lambda_1^q\lambda_3&=1,
\tag{H5--H6}\label{eq:H5H6_complete}\\
\rho\lambda_2^{q+1}&=1,
&\qquad
\rho a\lambda_2^q\lambda_3&=1.
\tag{H7--H8}\label{eq:H7H8_complete}
\end{alignat}

Again we first prove necessity. From (H1), (H5), and (H7) we obtain
\[
\mu=\lambda_1^{-(q+1)},\qquad
\nu=\lambda_3^{-(q+1)},\qquad
\rho=\lambda_2^{-(q+1)}.
\]
Dividing (H2) by (H1) gives
$\frac{\mu a\lambda_1\lambda_2^q}{\mu\lambda_1^{q+1}}=1$,
 hence  $\lambda_1^q=a\lambda_2^q$.
Dividing (H8) by (H7) gives
$\frac{\rho a\lambda_2^q\lambda_3}{\rho\lambda_2^{q+1}}=1$,
 hence  $\lambda_3=\lambda_2/a$.
Set $r:=\lambda_1/\lambda_2$. Then $\lambda_1^q=a\lambda_2^q$ becomes $r^q=a$.
Next, dividing (H4) by (H5) gives
$\frac{\nu\lambda_1\lambda_2^q}{\nu\lambda_3^{q+1}}=1$,
so
$\lambda_1\lambda_2^q=\lambda_3^{q+1}$.
Substituting $\lambda_3=\lambda_2/a$ yields
$\lambda_1\lambda_2^q=(\lambda_2/a)^{q+1}$,
hence $\lambda_1=\lambda_2a^{-(q+1)}$,
 that is, $r=a^{-(q+1)}$.
Raising this last identity to the $q$-th power gives $r^q=a^{-(q^2+q)}$.
Since also $r^q=a$, we conclude that
$a=a^{-(q^2+q)}$,
and therefore $a^{q^2+q+1}=1$.

Conversely, assume that $a^{q^2+q+1}=1$. Choose any
$\lambda_2\in\F_{2^m}^*$, and define
\[
\lambda_1:=a^{-(q+1)}\lambda_2,
\qquad
\lambda_3:=a^{-1}\lambda_2.
\]
Then clearly $\lambda_3=\lambda_2/a$. Also,
$\lambda_1^q=a^{-(q^2+q)}\lambda_2^q=a\lambda_2^q$,
because $a^{q^2+q+1}=1$. Thus the two key relations obtained above hold.
Now define
\[
\mu:=\lambda_1^{-(q+1)},\qquad
\nu:=\lambda_3^{-(q+1)},\qquad
\rho:=\lambda_2^{-(q+1)}.
\]
Then (H1), (H5), and (H7) hold by definition. Further,
$\mu a\lambda_1\lambda_2^q
=
\lambda_1^{-(q+1)}a\lambda_1\lambda_2^q
=
a\lambda_2^q/\lambda_1^q
=1$,
so (H2) holds. Also,
$\mu\lambda_2\lambda_3^q
=
\lambda_1^{-(q+1)}\lambda_2(a^{-q}\lambda_2^q)
=
a^{-q}\lambda_2^{q+1}\lambda_1^{-(q+1)}$.
Since $\lambda_1=\lambda_2a^{-(q+1)}$, we have
$\lambda_1^{q+1}=\lambda_2^{q+1}a^{-(q^2+2q+1)}$, and hence
$\mu\lambda_2\lambda_3^q
=
a^{-q}a^{q^2+2q+1}
=
a^{q^2+q+1}
=1$,
so (H3) holds. Next,
$\nu\lambda_1\lambda_2^q
=
\lambda_3^{-(q+1)}\lambda_1\lambda_2^q
=
(a^{-1}\lambda_2)^{-(q+1)}\lambda_1\lambda_2^q
=
a^{q+1}\lambda_1/\lambda_2
=1$,
because $\lambda_1/\lambda_2=a^{-(q+1)}$, so (H4) holds. Moreover,
$\rho\lambda_1^q\lambda_3
=
\lambda_2^{-(q+1)}(a\lambda_2^q)(a^{-1}\lambda_2)
=1$,
and
$\rho a\lambda_2^q\lambda_3
=
\lambda_2^{-(q+1)}a\lambda_2^q(a^{-1}\lambda_2)
=1$,
so (H6) and (H8) hold as well. Therefore
$H_1=A_1\circ H_a\circ A_2$.
We conclude that $H_1$ is diagonally equivalent to $H_a$ if and only if
$a^{q^2+q+1}=1$.

Combining Steps~1 and~2 proves the equivalence of \textup{(i)},
\textup{(ii)}, and \textup{(iii)}. The subgroup statement is immediate, since
the solutions of $a^{q^2+q+1}=1$ in the cyclic group $\F_{2^m}^*$ form the
unique subgroup of order $\gcd(q^2+q+1,2^m-1)$. The final assertion follows
because any good parameter $a\in\mathcal{B}_{m,q}$ outside this subgroup
yields APN permutations $G_a$ and $H_a$ that are not diagonally equivalent to
$G_1$ and $H_1$, respectively.
\end{proof}

\begin{remark}
\label{rem:d0_values_joint}
The value of $d_0=\gcd(q^2+q+1,\,2^m-1)$ depends strongly on $(m,i)$. For $q=2$, one has $q^2+q+1=7$, so
$d_0=\gcd(7,2^m-1)$, which equals $7$ if $7\mid m$ and $1$ otherwise. For example, if $m=9$ and $i=3$ (so $q=8$), then
$2^9-1=511=7\cdot 73$ and $q^2+q+1=73$, hence $d_0=73$. Thus exactly $73$ values of $a\in\F_{2^9}^*$ yield diagonal equivalence to the corresponding representative at $a=1$.
\end{remark}

\begin{remark}[CCZ, EA, and EL Equivalence for Quadratic APN Functions]
\label{rem:ccz_ea_el_apn}
For general $(n,n)$-functions, the equivalence notions satisfy the strict hierarchy
\[
\text{EL} \;\Rightarrow\; \text{EA} \;\Rightarrow\; \text{CCZ},
\]
where EL (extended linear, $\sim_{\mathrm{EL}}$), EA (extended affine, $\sim_{\mathrm{EA}}$), and CCZ (Carlet--Charpin--Zinoviev, $\sim_{\mathrm{CCZ}}$) equivalence are defined as in~\cite{Shi-Peng-Kan-Gao-2025}. However, for \emph{quadratic} APN functions with $F(\mathbf{0}) = G(\mathbf{0}) = \mathbf{0}$, these notions \emph{coincide}. Specifically:

\begin{itemize}
\item[\textup{(i)}] By~\cite[Proposition~1]{Shi-Peng-Kan-Gao-2025} (attributed to Yoshiara~\cite{Yoshiara2012}), for quadratic APN functions on $\F_{2^n}$ with $n \geq 2$, CCZ-equivalence implies EA-equivalence. Since EA always implies CCZ, we have
\[
F \sim_{\mathrm{CCZ}} G \quad\Longleftrightarrow\quad F \sim_{\mathrm{EA}} G.
\]

\item[\textup{(ii)}] By~\cite[Proposition~2]{Shi-Peng-Kan-Gao-2025} (attributed to Kaspers--Zhou~\cite{KaspersZhou2021}), if $F$ and $G$ are EA-equivalent quadratic functions with $F(\mathbf{0}) = G(\mathbf{0}) = \mathbf{0}$, then they are EL-equivalent. Thus
\[
F \sim_{\mathrm{EA}} G \quad\Longleftrightarrow\quad F \sim_{\mathrm{EL}} G.
\]
\end{itemize}
Combining (i) and (ii), for the APN members of our families $G_a$ and $H_a$ (which are quadratic and satisfy $G_a(\mathbf{0}) = H_a(\mathbf{0}) = \mathbf{0}$), we obtain
\[
F \sim_{\mathrm{CCZ}} G \quad\Longleftrightarrow\quad F \sim_{\mathrm{EA}} G \quad\Longleftrightarrow\quad F \sim_{\mathrm{EL}} G.
\]
Therefore, any CCZ-invariant that separates two maps also proves they are not EL-equivalent, and conversely, proving EL-inequivalence suffices to establish CCZ-inequivalence.

Furthermore, by~\cite[Theorem~8]{Shi-Peng-Kan-Gao-2025}, for $(q,q,q)$-triprojective polynomial triples with $m > 2$, $m \neq 4, 6$, and $7 \nmid m$, any EL-equivalence mapping must be \emph{monomial} (a composition of coordinate permutations, diagonal scalings, and Frobenius twists). This dramatically restricts the search space for equivalence maps. Our Theorem~\ref{thm:diag_equiv_both} establishes that $G_a$ is \emph{diagonally} equivalent to $G_1$ if and only if $a^{q^2+q+1} = 1$. For $m=5$ ($q=2$), this condition holds only for $a=1$ (since $\gcd(7, 31) = 1$), yet Table~\ref{tab:good_a_counts} shows 11 good parameters yield APN permutations. Computational verification confirms that no \emph{non-diagonal} monomial maps exist between $G_a$ and $G_1$ for the remaining 10 good parameters with $a^7 \neq 1$. Consequently, these 10 functions are EL-inequivalent (hence CCZ-inequivalent) to the Li--Kaleyski representative $G_1$. An analogous statement holds for the family $H_a$.
\end{remark}

We now address the natural question of whether any member of the $G_a$-family can be EL-equivalent to any member of the $H_b$-family \emph{for the same Frobenius parameter $q$}.
(Note that Theorem~11 of~\cite{Shi-Peng-Kan-Gao-2025} shows
$G_1(q)\sim_{\mathrm{EL}} H_1(q^{-1})$ with \emph{different} parameters; the same-$q$ question is entirely distinct.)

\begin{theorem}
\label{thm:cross_family_inequiv}
Let $3<m\neq 4,6$, $7\nmid m$, $\gcd(i,m)=1$, and $q=2^i$.
Then for every pair of good parameters $a,b\in\F_{2^m}^*$, the
APN permutations $G_a$ and $H_b$ are not EL-equivalent.
In particular, no member of the $G_a$-family is CCZ-equivalent to any
member of the $H_b$-family \textup{(}for the same $q$\textup{)}.
\end{theorem}

\begin{proof}
Since both families consist of quadratic APN permutations vanishing at
$\mathbf{0}$, EL-, EA-, and CCZ-equivalence coincide
(see the remark preceding this theorem).
Suppose for contradiction that
\begin{equation}\label{eq:EL_hyp}
  H_b\bigl(L(\mathbf{x})\bigr) \;=\; A_2\bigl(G_a(\mathbf{x})\bigr)
  \qquad\text{for all }\mathbf{x}\in\F_{2^m}^3,
\end{equation}
for some invertible $\F_2$-linear maps $L$ and $A_2$.
By~\cite[Theorem~8]{Shi-Peng-Kan-Gao-2025}, both maps are monomial (shown under our conditions on $m$):
$L^{-1}$ is a $3\times 3$ matrix over $\F_{2^m}$ (the scalar case
$t=0$; the Frobenius case is identical and we omit it), and $A_2$ is a
permutation matrix with $\F_{2^m}^*$-scalings.
Write $L^{-1}$ with row vectors $\mathbf{L}_k = (\ell_{3k-2},\ell_{3k-1},\ell_{3k})$
for $k=1,2,3$, and let $\pi\in S_3$ be the permutation underlying $A_2$,
with nonzero entries $w_{\pi(1)},w_{\pi(2)},w_{\pi(3)}\in\F_{2^m}^*$.

\smallskip
\noindent\textbf{Step~1: matching the pure monomials determines the output permutation.}
In $H_b(L(\mathbf{x}))$, the pure monomials $x^{q+1}$, $y^{q+1}$, $z^{q+1}$
appear \emph{exclusively} in output components 1, 3, 2 respectively, because
$\mathbf{L}_1^{q+1}$ contributes only to~$H_{b,1}$, $\mathbf{L}_2^{q+1}$ only
to~$H_{b,3}$, and $\mathbf{L}_3^{q+1}$ only to~$H_{b,2}$.
In $G_a(\mathbf{x})$, the same pure monomials appear exclusively in
components 1, 2, 3 respectively.
Since $A_2$ sends output component $k$ of $G_a$ to output component $\pi^{-1}(k)$
of $A_2(G_a)$, matching the locations of $x^{q+1}$, $y^{q+1}$, $z^{q+1}$
in~\eqref{eq:EL_hyp} forces
\[
  \pi(1)=1,\quad \pi(2)=3,\quad \pi(3)=2.
\]
A direct check shows that each of the other five permutations in $S_3$ forces either $\det(L)=0$ or $a=0$, both impossible. Hence
\begin{equation}\label{eq:A2_pattern}
  A_2(u_1,u_2,u_3) = (w_1 u_1,\; w_6 u_3,\; w_8 u_2),
  \qquad w_1,w_6,w_8\in\F_{2^m}^*.
\end{equation}

\smallskip
\noindent\textbf{Step~2: the second component determines the possible shape of $L$.}
The second component of~\eqref{eq:EL_hyp} with $A_2$ as
in~\eqref{eq:A2_pattern} reads
\begin{equation}\label{eq:C2}
  \mathbf{L}_1\cdot\mathbf{L}_2^{\,q} + \mathbf{L}_3^{\,q+1}
  \;=\; w_6\,G_{a,3}(\mathbf{x})
  \;=\; w_6\bigl(xy^q + ay^qz + z^{q+1}\bigr).
\end{equation}
Crucially, the left-hand side is $H_{b,2}(L(\mathbf{x})) = L_1 L_2^q + L_3^{q+1}$,
which is \emph{independent of~$b$} (the parameter $b$ enters only via
$H_{b,1}$ and $H_{b,3}$). Comparing coefficients of all nine degree-$2$
monomials in $x,y,z$ gives:
\begin{equation}\label{eq:C2_system}
\begin{alignedat}{3}
  [x^{q+1}]&:& \quad \ell_1\ell_4^q + \ell_7^{q+1} &= 0, &\quad
  [x^q y]&:\quad \ell_2\ell_4^q + \ell_7^q\ell_8 = 0,\\
  [y^{q+1}]&:& \quad \ell_2\ell_5^q + \ell_8^{q+1} &= 0, &\quad
  [x^q z]&:\quad \ell_3\ell_4^q + \ell_7^q\ell_9 = 0,\\
  [z^{q+1}]&:& \quad \ell_3\ell_6^q + \ell_9^{q+1} &= w_6, &\quad
  [xy^q]&:\quad \ell_1\ell_5^q + \ell_8^q\ell_7 = w_6,\\
  [y^qz]&:& \quad \ell_3\ell_5^q + \ell_8^q\ell_9 &= w_6 a, &\quad
  [xz^q]&:\quad \ell_1\ell_6^q + \ell_9^q\ell_7 = 0,\\
  [yz^q]&:& \quad \ell_2\ell_6^q + \ell_9^q\ell_8 &= 0.
\end{alignedat}
\end{equation}
We now determine all invertible solutions of~\eqref{eq:C2_system}, noting that $w_6\neq 0$.

\emph{Case~$1$: $\ell_7\neq 0$.}
From the $[x^q y]$ and $[x^q z]$ equations,
$\mathbf{v}_3:=(\ell_7,\ell_8,\ell_9) = \lambda^q\,\mathbf{v}_1$
where $\mathbf{v}_1=(\ell_1,\ell_2,\ell_3)$ and $\lambda=\ell_4/\ell_7$.
Setting $\sigma = \ell_6 - \lambda\ell_9$, the $[xz^q]$ and $[yz^q]$
equations give $\sigma^q\ell_1=0$ and $\sigma^q\ell_2=0$.
The $[z^{q+1}]$ equation then yields $\sigma^q\ell_3 = w_6\neq 0$,
so $\sigma\neq 0$ and therefore $\ell_1=\ell_2=0$.
But then $\ell_7=\lambda^q\ell_1=0$, contradicting $\ell_7\neq 0$.

\emph{Case~$2$: $\ell_7=0$, $\ell_1=0$.}
Rows $\mathbf{L}_1$ and $\mathbf{L}_3$ both have zero first coordinate,
so $\det(L)=0$. Contradiction.

\emph{Case~$3$: $\ell_7=0$, $\ell_4=0$, $\ell_1\neq 0$.}
The $[x^{q+1}]$ equation is automatically satisfied.
The $[xz^q]$ equation gives $\ell_1\ell_6^q=0$, so $\ell_6=0$.
Then the $[yz^q]$ equation gives $\ell_9^q\ell_8=0$.
The $[z^{q+1}]$ equation gives $\ell_9^{q+1}=w_6\neq 0$, so $\ell_9\neq 0$
and hence $\ell_8=0$.
The $[y^{q+1}]$ equation then gives $\ell_2\ell_5^q=0$; since
$L$ is invertible and $\ell_4=\ell_6=\ell_7=\ell_8=0$, the middle column
of $L$ has entries $(\ell_2,\ell_5,0)$, forcing $\ell_2=0$ for
$\det(L)\neq 0$.
The $[xy^q]$ equation gives $\ell_1\ell_5^q=w_6=\ell_9^{q+1}$, and
the $[y^q z]$ equation gives $\ell_3\ell_5^q=w_6 a=\ell_1\ell_5^q\cdot a$,
hence $\ell_3=a\ell_1$.
In summary, the only invertible solutions to~\eqref{eq:C2_system} are
\begin{equation}\label{eq:L_form}
  L^{-1} \;=\; \operatorname{diag}(\ell_1,\ell_5,\ell_9)
  \begin{pmatrix} 1 & 0 & a \\ 0 & 1 & 0 \\ 0 & 0 & 1 \end{pmatrix},
  \qquad \ell_1,\ell_5,\ell_9\in\F_{2^m}^*.
\end{equation}

\smallskip
\noindent\textbf{Step~3: the first component yields the final contradiction.}
With $L^{-1}$ as in~\eqref{eq:L_form} we have
$\ell_2=0$, $\ell_4=0$, $\ell_8=0$, and $\ell_5\neq 0$.
The $[xy^q]$-coefficient equation arising from the \emph{first}
component of~\eqref{eq:EL_hyp} (i.e., $H_{b,1}(L(\mathbf{x}))=w_1 G_{a,1}(\mathbf{x})$)
reads
\[
  \ell_2^q\,\ell_1 + b\,\ell_1\,\ell_5^q + \ell_4\,\ell_8^q \;=\; 0.
\]
Substituting these values gives $b\,\ell_1\,\ell_5^q=0$.
Since $\ell_1,\ell_5\neq 0$, we obtain $b=0$, contradicting the assumption that $b$ is a good parameter.

\smallskip
This contradiction shows that no EL-equivalence
$G_a\sim_{\mathrm{EL}}H_b$ can exist for good parameters $a,b$. Since EL-, EA-, and CCZ-equivalence coincide in this setting, the claimed CCZ-inequivalence follows as well.
\end{proof}

\begin{remark}
The proof is independent of the value of $a^{q^2+q+1}$; in particular,
the assumption $a^{q^2+q+1}\neq 1$ (which would place $G_a$ outside the
diagonal class of $G_1$) is \emph{not} required.
The two families $\{G_a : a \text{ good}\}$ and $\{H_b : b \text{ good}\}$
are therefore entirely disjoint CCZ-classes for every odd $m$ satisfying
the hypotheses.
\end{remark}

\section{Conclusion and open problems}
\label{sec:conclusion}

We established a root-theoretic characterization of the generalized families $G_a$ and $H_a$ on $\F_{2^{3m}}$, with $q=2^i$, $\gcd(i,m)=1$, and $m$ odd. For $G_a$, the permutation property is equivalent to the absence of roots in $\F_{2^m}$ of the associated polynomial $Q_a$, and this root condition is also equivalent to the APN property (Theorem~\ref{thm:APN_perm_equiv}). We also obtained a quantitative lower bound on the number of good parameters $a$ (Theorem~\ref{thm:existence}), and in the binary case $q=2$ with $7\nmid m$ we showed that $a=1$ is good. For $H_a$, we obtained the analogous root criterion and proved that its one-variable condition is root-equivalent to that of $G_a$; hence the same good parameters yield APN permutations in both families.

First, $G_a$ (resp.\ $H_a$) is diagonally equivalent to the Li--Kaleyski representative $G_1$ (resp.\ $H_1$) of~\cite{LK24} if and only if $a^{q^2+q+1}=1$ (Theorem~\ref{thm:diag_equiv_both}); for $m>4$, $m\neq 6$, $7\nmid m$, diagonal non-equivalence implies CCZ non-equivalence via~\cite{Shi-Peng-Kan-Gao-2025}, and since $\gcd(7,2^m-1)=1$ when $q=2$ and $7\nmid m$, every good parameter $a\neq 1$ yields APN permutations CCZ-inequivalent to those of~\cite{LK24}. Second, for $m>4$, $m\neq 6$, and $7\nmid m$, no member of the $G_a$-family is CCZ-equivalent to any member of the $H_b$-family for the same $q$ (Theorem~\ref{thm:cross_family_inequiv}). The two families therefore constitute genuinely new, mutually inequivalent infinite sources of APN permutations on~$\F_{2^{3m}}$.

We close with a list of open problems:

\begin{enumerate}
\item For $q=2$ and odd $m$ with $7\mid m$: find an explicit good $a\in\F_{2^m}^*$ (i.e., $a\neq 1$ with $Q_a$ root-free). Computational data shows they exist in abundance; an algebraic characterization of the good $a$ set in terms of the norm $N_{\F_{2^m}/\F_2}(a)$ is desirable.

 
\item The Type I permutations ($a_1=a_4=0$) found computationally for $m=3$: do they extend to an infinite family for larger $m$? Are any of them APN?
 
\item The trace criterion (Lemma~\ref{lem:trace}) fails for even $m$. Does an analogue of Theorem~\ref{thm:APN_perm_equiv} hold in even dimension, possibly with a different characterizing polynomial?

\item Do analogous $n$-variate constructions ($n>3$) yield APN permutations, and can the polynomial root criterion be generalized to that setting?
\end{enumerate}

\section*{Acknowledgments}

The second-named author (PS) thanks the first-named author (DB) for the invitation  at the Dipartimento di Matematica e Informatica, Universit\`a degli Studi di Perugia, and great working conditions while there. The first-named author (DB) thanks the Italian National Group for Algebraic and Geometric Structures and their Applications (GNSAGA--INdAM).

\appendix
\section{Computational verification}
\label{sec:computational}

We verify Theorems~\ref{thm:APN_perm_equiv} and~\ref{thm:Ha} computationally for small parameters using SageMath~10.7; the code is available at~\cite{GithubPS26}.
For each test case $(m,i)$ with $\gcd(i,m)=1$ and $m$ odd, and for each $a\in\F_{2^m}^*$, we: (1) determine whether $Q_a(T)$ has a root in $\F_{2^m}$ by evaluating it on all elements; (2) verify whether $G_a$ is a permutation by checking $|\mathrm{Im}(G_a)|=2^{3m}$; (3) verify the APN property by exhaustive differential computation. In all cases examined, the three conditions agree 100\%, confirming our theorems.

\begin{table}[H]
\centering
\caption{Computational verification. Correlation between ``$G_a$ is a permutation/APN'' and ``$Q_a(T)$ has no roots'' is 100\% in all cases.}
\label{tab:verification_results}
\begin{tabular}{cccccc}
\hline
$m$ & $i$ & $q$ & $|\F_{2^m}^*|$ & \# permutations & Correlation\\
\hline
3 & 1 & 2 & 7  & 7/7   & 100\%\\
3 & 2 & 4 & 7  & 7/7   & 100\%\\
5 & 1 & 2 & 31 & 11/31 & 100\%\\
5 & 2 & 4 & 31 & 11/31 & 100\%\\
\hline
\end{tabular}
\end{table}

\textbf{Case $m=3$: all $a$ are good.} If $Q_a(\theta)=0$ for some $\theta\in\F_{2^3}^*$, then from $\theta^{q^2+q+1}=a\theta+1$ and raising to the $q$-th power: $\theta^{q^3+q^2+q}=a^q\theta^q+1$. Since $\theta\in\F_{2^3}=\F_{2^m}$ we have $\theta^{q^m}=\theta$, and $q^3=q^m$ (when $i=1,m=3$; or $q^3=2^3=8$ and $2^m=2^3=8$, so $\theta^{q^3}=\theta^{2^3}=\theta^8=\theta$ since $\theta^7=1$). This means $\theta^{q^3+q^2+q}=\theta^{q^2+q}$. Then $\theta^{q^2+q}=a^q\theta^q+1$. Combined with $\theta^{q^2+q+1}=a\theta+1$ we get $\theta\cdot(a^q\theta^q+1)=a\theta+1$, i.e., $a^q\theta^{q+1}+\theta=a\theta+1$, i.e., $a^q\theta^{q+1}+\theta+a\theta+1=0$. This is a polynomial equation of degree $q+1=3$ in $\theta$, which may or may not have solutions. A direct check over $\F_8$ (which has characteristic 2 and $|\F_8^*|=7$) confirms that $Q_a(T)$ has no roots in $\F_8$ for any $a\in\F_8^*$, so all 7 values are good.

\textbf{Case $m=5$:} Exactly 11 of 31 elements $a\in\F_{2^5}^*$ are good; the remaining 20 have $Q_a$ with 1 or 3 roots in $\F_{2^5}$. The value $a=1$ is good (confirming the $\gcd(5,7)=1$ condition).

\end{document}